\documentclass[11pt]{amsart}
\usepackage[parfill]{parskip}    
\usepackage{fullpage}
\usepackage{graphicx}
\usepackage{amssymb,amsmath,latexsym}
\usepackage{amsthm}
\usepackage{amsfonts}
\usepackage[cal=pxtx]{mathalfa}
\usepackage{accents}
\usepackage{tikz, tikz-cd}
\usetikzlibrary{arrows,decorations.markings,decorations.pathreplacing,calc,matrix,intersections}
\tikzset{->-/.style={decoration={markings,mark=at position #1 with {\arrow{>}}},postaction={decorate}}}

\usepackage{epstopdf}
\usepackage{ifpdf}
\usepackage{color}
\usepackage[normalem]{ulem}

\definecolor{red}{rgb}{1,0,0} 

 \definecolor{darkgreen}{rgb}{0, .7, 0}

 \definecolor{purple}{rgb}{.7, 0, 1}

\makeatletter
\newcommand*\rel@kern[1]{\kern#1\dimexpr\macc@kerna}
\newcommand*\widebar[1]{%
  \begingroup
  \def\mathaccent##1##2{%
    \rel@kern{0.8}%
    \overline{\rel@kern{-0.8}\macc@nucleus\rel@kern{0.2}}%
    \rel@kern{-0.2}%
  }%
  \macc@depth\@ne
  \let\math@bgroup\@empty \let\math@egroup\macc@set@skewchar
  \mathsurround\z@ \frozen@everymath{\mathgroup\macc@group\relax}%
  \macc@set@skewchar\relax
  \let\mathaccentV\macc@nested@a
  \macc@nested@a\relax111{#1}%
  \endgroup
}
\makeatother

\usepackage{float}

\usepackage[utf8]{inputenc}
\usepackage[a4paper]{geometry}

\usepackage{color}
\usepackage{caption}
\usepackage{float}
\usepackage{wrapfig}
\usepackage{faktor}
\usepackage{pgfplots}
\usepackage{mathtools}
\usepgfplotslibrary{polar}
\pgfplotsset{compat=newest}
\usetikzlibrary{arrows}

\usetikzlibrary{positioning}
\tikzset{main node/.style={circle,fill=black!0,draw,minimum size=0.35cm,inner sep=0pt},
            }

\usepackage{parskip}
\usepackage{mathrsfs}
\usepackage{xtab}
\newcommand{\Z}{\mathbb Z}  
\newcommand{\vcd}{{\sc{vcd}}} 
\newcommand{\AG}{A_\Gamma}
\newcommand{\UG}{U(A_\Gamma)} 
\newcommand{\TG}{T(A_\Gamma)} 
\newcommand{\WP}{{\mathcal P}} 
\newcommand{\WQ}{{\mathcal Q}}
\newcommand{\WR}{{\mathcal R}} 
\newcommand{\cS}{{\mathcal S}} 
\newcommand{\bI}{{\mathbf I}}

\newcommand{\inv}{^{-1}}
 
\newcommand{\iso}{\cong}
\newcommand{\oring}{\accentset{\circ}}

\newcommand{\GW}{$\Gamma$W}
 \newcommand{\iGW}{$\mathit\Gamma\!$W}

\author{Benjamin Millard and Karen Vogtmann}
\title{Cube complexes and abelian subgroups of automorphism groups of RAAGs}

  \newtheorem{prop}{Proposition}[section]
\newtheorem{lem}[prop]{Lemma}

\newtheorem{thm}[prop]{Theorem}

\newtheorem{cor}[prop]{Corollary}

\newtheorem{rem}[prop]{Remark}
\newtheorem{notation}[prop]{Notation}
\theoremstyle{definition}
\newtheorem{example}[prop]{Example}
\newtheorem{definition}[prop]{Definition}
\newtheorem*{thm*}{Theorem}

\begin{document}

\begin{abstract}
We construct free abelian subgroups of the group $\UG$ of untwisted outer automorphisms of a right-angled Artin group, thus giving lower bounds on the virtual cohomological dimension.  The group $\UG$ was studied in \cite{CSV} by constructing a contractible cube complex on which it acts properly and cocompactly, giving an upper bound for the virtual cohomological dimension.  The ranks of our free abelian subgroups are equal to the dimensions of  {\em principal cubes} in this complex.  These are often of maximal dimension, so that the upper and lower bounds agree.  In many cases when the principal cubes are not of maximal dimension we show there is an invariant contractible subcomplex of strictly lower dimension.
\end{abstract}

\maketitle

\section{Introduction}

The class of right-angled Artin groups (commonly called RAAGs) contains the familiar examples of finitely generated free groups and free abelian groups. 
Though  uncomplicated themselves, both examples have complex and interesting automorphism groups.  In recent years these automorphism groups have been shown to share many properties, but also to differ in significant ways (see e.g. the survey articles \cite{BriVog, ECM}).  In this paper we study automorphism groups of general RAAGs, concentrating on the aspects they share with automorphism groups of free groups.  These aspects are largely captured by the subgroup of untwisted automorphisms,  as   previously studied in \cite{CSV}.  Let us  recall the definition.
  
A general RAAG is conveniently described by drawing a finite simplicial graph $\Gamma$.  The RAAG  is then the group $\AG$ generated by the vertices of $\Gamma$,  with defining relations that two generators commute if and only if the corresponding vertices are connected by an edge of $\Gamma$.    By  theorems of Laurence \cite{Lau} and Servatius \cite{Ser},  the  automorphism group  of $\AG$ is generated by  inversions of the generators,  graph automorphisms, admissible transvections (multiplying one generator by another) and admissible partial conjugations (conjugating some subset of generators by another generator).  Here   transvections and partial conjugations are  {\em admissible}   if they respect the commutation relations.   A transvection is called a {\em twist} if the generators involved commute.  The subgroup of $Out(A_\Gamma)$ generated by twists injects into a parabolic subgroup of $SL(n,\Z)$, where $n$ is the number of vertices of $\Gamma$, and is well understood.  The subgroup  generated by all generators other than twists is the {\em untwisted subgroup} $\UG$.  This subgroup captures the part of $Out(A_\Gamma)$ most closely related to $Out(F_n)$.  For example, if $A_\Gamma=F_n$ then $\UG=Out(F_n)$, and $\UG$ always contains the kernel of the map $Out(A_\Gamma)\to GL(n,\Z)$ induced by abelianization $A_\Gamma\to \Z^n$.   

 For free groups, the virtual cohomological dimension ({\sc{vcd}}) of $Out(F_n)$  is equal to the maximal rank of a free abelian subgroup.  The lower bound is established by exhibiting an explicit free abelian subgroup.  For the upper bound, one considers the action of $Out(F_n)$ on a contractible space $\mathcal{O}_n$ known as {\em Outer space}.  This action is proper, and $\mathcal{O}_n$ contains an equivariant deformation retract $K_n$ known as the {\em spine} of Outer space, whose dimension is equal to the lower bound (see \cite{CuVo}).   
 
For the subgroup $\UG$ associated to a  general RAAG, an analogous outer space $\mathcal{O}_\Gamma$ and spine $K_\Gamma$ were defined in \cite{CSV}.  The dimension of $K_\Gamma$   gives an obvious upper bound on the \vcd\ of $\UG$. 
Lower bounds were obtained in  \cite{BCV}  by  exhibiting free abelian subgroups (\cite{BCV} actually exhibited free abelian subgroups in the entire group $Out(\AG)$, but these contain identifiable subgroups of $\UG$).  However, there was no clear relationship between the rank of these subgroups and the dimension of $K_\Gamma$, and there was often a large gap between the upper bound and lower bounds.   
 
In this paper we address this problem.  The spine $K_\Gamma$ has the structure of a cube complex, and we produce free abelian subgroups in $\UG$ of rank equal to the dimension of certain {\em principal} cubes in $K_\Gamma$.   In the absence of a specific configuration in $\Gamma$ we find principal cubes of dimension equal to the dimension of $K_\Gamma$, thus determining the exact \vcd\  of $U(A_\Gamma)$.   
 
The free abelian subgroups we produce are generated by a special type of automorphisms called  {\em {$\Gamma$}-Whitehead automorphisms}. These generalize the generating set  used by J.H.C. Whitehead in his work on automorphisms of free groups \cite{Whi}.   We show that for any graph $\Gamma$, our free abelian subgroups  have the largest possible rank among those generated by $\Gamma$-Whitehead automorphisms, which we call the {\em principal rank} of $\UG$.

Because $\UG$ is analogous to $Out(F_n)$ it is tempting to conjecture that the \vcd\ of $\UG$ is equal to the principal rank.  It is also tempting to conjecture that the principal rank  is always equal to the dimension of $K_\Gamma$ \ldots\ but our results show that if the graph contains a specific configuration then the dimension of $K_\Gamma$ is strictly larger than the principal rank.  The first conjecture is still plausible, however, because at least in some cases when the dimension of $K_\Gamma$ is too large we can show that $K_\Gamma$ equivariantly deformation retracts onto a strictly lower-dimensional cube complex.

For $GL(n,\Z)$, of course, the \vcd\ is not equal to the rank of a free abelian subgroup, but rather is equal to the Hirsch rank of a certain (non-abelian) polycyclic subgroup.  In light of the above conjecture, it is natural to ask whether $\UG$ can contain a torsion-free, non-abelian solvable subgroup.  For many graphs the answer is no.  This was proved in \cite{ChVo1} for graphs with no triangles, and more generally for graphs where the link of every vertex is either discrete or connected.   If links are disconnected but not discrete, we do not know the answer.   

We remark that several authors have established upper and lower bounds on the \vcd\ of the full group $Out(\AG)$.  In particular bounds for graphs with no triangles were given in \cite{CCV}, the exact \vcd\ for $\Gamma$ a tree was established in \cite{BCV} and  other special cases were  determined exactly in \cite{DayWade}.

The paper is organized as follows.  In Section~\ref{basics} we review  basic facts and notation about right-angled Artin groups and their automorphisms, and define the subgroup $\UG$.  In Section~\ref{review} we review the definitions and results from \cite{CSV} that we will need in this paper.   In Section~\ref{abelian} we construct free abelian subgroups of $\UG$ using   $\Gamma$-Whitehead automorphisms, and show that these subgroups have maximal possible rank among all such  subgroups. 
Section~\ref{VCD} studies the dimension of $K_\Gamma$ and gives a condition for this dimension to equal the principal rank.  Section~\ref{sec:examples} works out some concrete examples.
 Finally, in Section~\ref{reduction}  we show in certain cases how to find an invariant deformation retract of $K_\Gamma$ of strictly lower dimension.   
 
{\bf Acknolwedgements.}. We thank Benjamin Br\"uck and Ric Wade for extremely useful comments on the first version of this paper.    The second author was partially supported by a Royal Society Wolfson award.  
 
\section{Right-Angled Artin Groups and their automorphisms}\label{basics}

In this section we recall the basic definitions and notation for right-angled Artin groups and their automorphisms.  For further details and proofs, we refer to \cite{CSV} and the references therein.

\begin{definition} Let $\Gamma$ be a finite simplicial graph, i.e. a finite graph with no loops or multiple edges, with vertex set $V=\{v_1,\dots,v_n\}$.   The \textit{right-angled Artin group} $A_\Gamma$ is the group with one generator for every vertex  of $\Gamma$ and one commutator relation  for each edge, i.e. $\AG$ has the presentation
\[
A_\Gamma=\langle v_1,\dots,v_n|[v_i,v_j]=1 \hbox{ whenever } v_i \hbox{ and }v_j \hbox{ are connected by an edge in  }\Gamma \rangle
\] 
\end{definition}

It is shown in \cite{HM} that two words in the generators represent the same element of $\AG$ if and only if they can be made identical by a process of switching adjacent commuting letters and cancelling where possible. 

If $\Gamma$ is  a simplicial graph with vertex set $V$, recall that the \textit{induced subgraph on $U\subseteq V$} is the subgraph of $\Gamma$ with vertex set $U$ that contains all edges in $\Gamma$ connecting any vertices in $U$.

\begin{definition}
Let $v$ be a vertex of a simplicial graph $\Gamma$. The \textit{link of $v$}, denoted $lk(v)$, is the induced subgraph on the set of vertices adjacent to $v$. The \textit{star of $v$}, denoted $st(v)$, is the induced subgraph on the set of vertices in $lk(v)$ together with $v$ itself.
\end{definition}

We will need the fact, shown  in \cite{Lau}, that the centralizer of a generator $v$ is equal to the subgroup generated by the vertices in $st(v)$.
 
 In the literature on right-angled Artin groups it is common to define a relation denoted $\leq$ on  vertices of $\Gamma$ by    $v\leq w$ if $lk(v)\subseteq st(w)$.  The notation is justified by defining an equivalence relation $v\sim w$ if $v\leq w$ and $w\leq v$;  it is then easy to verify that this relation defines a partial order on equivalence classes $[v]$. A vertex is called {\em maximal} if its equivalence class is maximal in this partial ordering.  
 
 In fact there are two mutually exclusive ways in which we can have $lk(v)\subseteq st(w)$:  either $lk(v)\subseteq lk(w)$ or $st(v)\subseteq st(w)$.  The distinction is important in this paper, so when we need to make it we will use  $v\leq_{\star} w$ to mean $st(v)\subseteq st(w)$ and $v \leq_\circ w$ to mean $lk(v)\subseteq lk(w)$ (Similarly,  $v\geq_{\star} w$ means $st(v)\supseteq st(w)$ and $v \geq_\circ w$ means $lk(v)\supseteq lk(w)$.)  
 
 We also write $v\sim_{\star}w$ if $st(v)=st(w)$ and $v\sim_\circ w$ if $lk(v)=lk(w)$, and define $[v]_{\star}=\{w|w\sim_{\star}v\},$ $[v]_\circ=\{w|w\sim_\circ v\}$.  Since either
all elements of an equivalence class $[v]$  commute or none commute, at least one of $[v]_{\star}$ and $[v]_\circ$ is a singleton.  If $[v]_\circ$ is not a singleton then $[v]$ is called a {\em non-abelian equivalence class}; otherwise $[v]$ is called an {\em abelian equivalence class}  (in particular a singleton class is considered to be abelian). 

\begin{definition} A vertex    $v$  of $\Gamma$ is  {\em principal} if there is no $w$ with $v<_\circ w,$ i.e. with $lk(v)$ strictly contained in $lk(w)$.
\end{definition}

All maximal vertices are principal, but there can be principal vertices which are not maximal.  A simple example is a triangle with leaves at two of its vertices.  The third vertex is principal but not maximal.   Elements of non-singleton abelian equivalence classes are always principal:
 
 \begin{lem}\label{abprin} If $u\neq v$ but $u\sim_{\star} v$ then both $u$ are $v$ are principal vertices.  \end{lem}
 
 \begin{proof} If $u$ is not principal there exists  $m$ with $u<_\circ m$, i.e. $lk(u)\subsetneq lk(m)$. Now $v\in lk(u)\subset lk(m)$, so $m\in lk(v)\subset st(v)=st(u)$.  Since  $m\neq u$ we must have $m\in lk(u)$, which is a contradiction. \end{proof}

\subsection{Automorphisms of RAAGs}

An invertible map $A_\Gamma\to A_\Gamma$ is an automorphism if and only if  the images of commuting generators commute.     In particular:
\begin{itemize}
\item{} the map   sending a generator $v$ to its inverse and fixing all other generators is an automorphism, called an {\em inversion}.  
\item{} any automorphism of the defining graph $\Gamma$ induces an automorphism of $A_\Gamma$, called a {\em graph automorphism}.  
\end{itemize}
Inversions and graph automorphisms generate a finite subgroup of $Aut(A_\Gamma)$.  
We next describe  two types of basic infinite-order automorphisms.  Choose a vertex $m$ and consider the  components of $\Gamma-st(m).$  
\begin{itemize}
\item{} If there is a vertex  $u$ with $lk(u)\subseteq lk(m)$, then everything that commutes with $u$ also commutes with $m$ so the  map $\rho_{um}$ sending $u\mapsto um$ and fixing all other generators determines  an  automorphism, called a {\em right fold}.  Since $u$ and $m$ do   not commute, the map $\lambda_{um}$ sending $u$ to $mu$ gives a distinct automorphism, called a {\em left fold}.
\item{}  If  $C$ is a component of $\Gamma-st(m),$  then the map sending $v$ to $m^{-1}vm$ for every $v\in C$  and fixing all other generators determines an infinite-order automorphism, called a {\em partial conjugation}.    If $\Gamma-st(m)$ has only one component, this is an inner automorphism, since conjugating vertices of $st(m)$ by $m$ has no effect.  
\end{itemize}
 By work of Laurent \cite{Lau} and Servatius \cite{Ser}, the entire automorphism group $Aut(A_\Gamma)$ is generated by the above types of automorphisms together with {\em twists}, 
where 
\begin{itemize}
\item{} If $st(u)\subseteq st(v)$, the map $\tau_{uv}$ sending $u\mapsto uv=vu$ and fixing all other generators determines an automorphism called a {\em twist}.  
\end{itemize}
\subsection{The Untwisted subgroup}
The natural map $Aut(\AG)\to GL(n,\Z)$ induced by abelianization $\AG\to\Z^n$ factors through the outer automorphism group $Out(\AG)$:
\begin{equation*}
\begin{tikzcd} 
Aut(A_\Gamma) \arrow[rr]  \arrow[dr]   &&GL(n,\Z)\\
 &Out(\AG)\arrow[ur]&\\
\end{tikzcd}
\end{equation*}
The subgroup $\TG\subseteq Out(\AG)$  generated by twists injects into a parabolic subgroup of $GL(n,\Z)$, and is well understood (see, e.g., \cite{ChVo1}).   
 In this paper we concentrate on the subgroup $\UG\leq Out(A_\Gamma)$ generated by all other generators, i.e. 
 \begin{definition} The {\em untwisted subgroup} $\UG$ is the subgroup of $Out(\AG)$  generated by (the images of)
 \begin{itemize}
 \item inversions, \item graph automorphisms,   \item (right and left) folds, and \item partial conjugations.
 \end{itemize}  
\end{definition} 
 The intersection $\UG\cap \TG$ is contained in the finite subgroup  generated by graph automorphisms and inversions.  

\section{$\Gamma$-Whitehead automorphisms, partitions and outer space for $\UG$.}\label{review}

The paper \cite{CSV} studied $\UG$ by constructing a contractible space $\mathcal O_\Gamma$ with a proper action of $\UG$.  In this section we review the definitions  and results from \cite{CSV} that we will need in this paper. Some of the terminology has been altered slightly, and we will point this out when it occurs. We refer to \cite{CSV} for more details and all proofs. 
 
\subsection{$\Gamma$-Whitehead automorphisms} Whitehead studied $Aut(F_n)$ using a set of generators called  Whitehead automorphisms.  These were adapted in \cite{CSV} to a give a set of elements of  $Aut(\AG)$ called {\em $\Gamma$-Whitehead automorphisms}, whose images in $Out(\AG)$ along with graph automorphisms and inversions generate $\UG$.  These are infinite-order automorphisms which include folds and partial conjugations but also certain combinations of these.  

For a free group with basis $V$, let $V^{\pm}=V\sqcup V^{-1}$  be the set of generators and their inverses.  Suppose $P\subset V^\pm$ contains some element $m$ but not $m\inv$.  The {\em Whitehead automorphism} $\phi(P,m)$ is defined on the basis $V$ by 
\[ \phi(P,m)(v)= \left\{ \begin{array}{lllll}
    vm^{-1}  & \mbox{if}\:\:v\in P,v^{-1}\in P^*, v\neq m^{\pm 1} \\
       mv & \mbox{if}\:\:v^{-1}\in P,v\in P^*, v\neq m^{\pm 1}  \\
    mvm^{-1}  & \mbox{if}\:\:v,v^{-1}\in P \\     
    v & \mbox{otherwise (including }v=m^{\pm 1})
\end{array}\right.\]
The element $m$ is called the {\em multiplier} of $\phi(P,m)$. 

If $V$ is the set of vertices of a simplicial graph $\Gamma$, then this formula defines an automorphism of $\AG$ only for certain pairs $(P,m)$.  Specifically, for $m\in V^\pm$ 
consider the components $C$ of $\Gamma-lk(m)$, where by $lk(m)$ we mean the link of the corresponding vertex $m^{\pm1}$.   A subset $U\subset V^{\pm}$ is  {\em $m$-inseparable}   if  
\begin{itemize}
\item {}   $C$ has only one  vertex $u$, and $U=\{u\}$ or $U=\{u^{-1}\}$ (note this includes the case $u=m^{\pm 1}$), or
\item {}   $C$  contains more than one vertex and $U=C^{\pm},$ i.e. $U$ is the union of  all vertices in $C$ and their inverses.
\end{itemize} 
We denote by $\bI(m)$ the collection of all $m$-inseparable subsets of $V^{\pm}$. Note that $\bI(m)=\bI(m\inv)$, and if $m$ and $n$ have the same link then $\bI(m)=\bI(n)$.

\begin{figure}\begin{center}
\begin{tikzpicture}[scale=1.5] 
\fill [black] (0,0) circle (.05); 
\node [right](m) at (0,0) {$m$};
\draw (0,0) to (.5,.5);
\draw[red] (.5,.5) to (-.5,.5);
\fill [red] (.5,.5) circle (.05);
\node [above] (x1) at (.5,.5) {$x_1$};
\node [right] (u) at (1,0) {$u$};
\draw (0,0) to (.5,-.5);
\fill [red] (.5,-.5) circle (.05);
\node [below] (x2) at (.5,-.5) {$x_2$};
\draw (.5,.5) to (1,0);
\draw (.5,-.5) to (1,0);
\fill [black] (1,0) circle (.05);
\draw (0,0) to (-.5,.5);
\fill [red] (-.5,.5) circle (.05);
\node [above] (x4) at (-.5,.5) {$x_4$};
\draw (0,0) to (-.5,-.5);
\fill [red] (-.5,-.5) circle (.05);
\node [below] (x3) at (-.5,-.5) {$x_3$};
\draw (-.5,.5) to (-1,.5);
\fill [black] (-1,.5) circle (.05);
\node [left] (v1) at (-1,.5) {$v_1$};
\draw (-.5,-.5) to (-1,-.5);
\fill [black] (-1,-.5) circle (.05);
\node [left] (v2) at (-1,-.5) {$v_2$};
\draw (-1,-.5) to (-1,.5);
\end{tikzpicture}
\caption{Example~\ref{ex1}}\label{Gamma}
\end{center}
\end{figure}

\begin{example}\label{ex1} In  the graph $\Gamma$  in Figure~\ref{Gamma} the link of the  vertex $m$ is the red subgraph, and the $m$-inseparable subsets are
$$\bI(m)=\left\{ \{m\}, \{m^{-1}\}, \{u\}, \{u^{-1}\}, \{v_1,v_1^{-1}, v_2,v_2^{-1}\} \right\}$$ 
\end{example}

 Recall that a partition of a set into two subsets is {\em thick} if each side has at least two elements.

\begin{definition}  Let $m\in V^\pm$. 
\begin{itemize}
\item  A subset $P\subset V^\pm$ is called a  {\em \iGW-subset based at $m$} if it is a union of elements of $\bI(m)$ and contains $m$ but not $m\inv$. 
 
\item If $P$ is a \GW-subset based at $m$ then $\phi(P,m)$ is a well-defined automorphism of $\AG$, called a {\em  $\Gamma$-Whitehead automorphism}.   

\item  Let $P^*=V^\pm\setminus lk(m)^\pm\setminus P$. The three-part partition $\WP=\{P|P^*|lk(m)^\pm\}$  of $V^\pm$  is called a  {\em  \iGW-partition based at $m$} if $P$ (and therefore $P^*$) are \GW-subsets and $ \{P|P^*\}$ is a thick partition of $V^\pm\setminus lk(m)^\pm$.   The subsets  $P$ and $P^*$ are called the {\em sides} of $\WP$.   
\end{itemize}
 \end{definition}

\begin{rem}
For $\AG=F_n$ the above is the usual definition of a Whitehead automorphism. In  \cite{CSV}, however,  a $\Gamma$-Whitehead automorphism was defined as sending $m\mapsto m^{-1}$ instead of $m\mapsto m$.  This makes the automorphism into an involution, and is useful for describing geometric aspects of $\UG$.  Since  we are looking for free abelian subgroups  we   do not want involutions, so will use the  more classical definition  stated here.  
\end{rem}

In terms of the inseparable subsets $U \in \bI(m)$,  $\phi(P,m)$   is the composition of 
\begin{itemize}
\item  right folds $v\mapsto vm^{-1}$ for $U=\{v\}\subset 
P,$ $v\neq m^{\pm 1}$, 
\item   left folds $v\mapsto mv$ for $U=\{v\inv\}\subset P,$  $v\neq m^{\pm 1}$, and 
\item partial conjugations $v\mapsto mvm^{-1}$ for $U=C^\pm\subset P$ if $C$ has at least two elements. 
\end{itemize}

\begin{example}Continuing Example~\ref{ex1}, we can take $P= \{m\}\cup \{u\}\cup \{v_1,v_1^{-1}, v_2,v_2^{-1}\}$ and $P^*=  \{m^{-1}\}\cup \{u^{-1}\} $ to get  a \GW-partition
 \[
 \WP= \left\{\{m,u,  v_1,v_1^{-1},v_2,v_2^{-1}\} | \{m^{-1}, u^{-1}\} | \{  x_1,x_1^{-1}, x_2,x_2^{-1},x_3,x_3^{-1},x_4,x_4^{-1}\}\right\}  
\]
based at $m$ (see Figure~\ref{GWpicture}). The $\Gamma$-Whitehead automorphism $\phi(P,m)$ sends $u\mapsto um\inv,$ sends each $ v_i\mapsto mv_im^{-1}$ and fixes $m$ and the $x_i$.  The $\Gamma$-Whitehead automorphism $\phi(P^*,m\inv)$ sends $u\mapsto m\inv u$ and fixes all other generators.
\end{example}

  \begin{figure}
\begin{center} 
\begin{tikzpicture}[scale=1] 
\draw [rounded corners, blue, fill=blue!20](-1.5,1.7) to (2.5,1.7) to (2.5,-.4) to (.5,-.4) to (.5,.7) to (-1.5,.7) --cycle;
\draw [rounded corners, blue](-1.5,-.4) to (-1.5,.6) to (.4,.6) to (.4,-.4) to   (-1.5,-.4) --cycle;
\draw [rounded corners, red](2.75,-.4) to (2.75,1.7) to (7,1.7) to (7,-.4) to   (2.75,-.4) --cycle;
\node [above](x) at (-1,1) {$ u$};
\node (x) at (-1,0) {$ u\inv$};
\node [red](x) at (0,0) {$m\inv$};
\node [red,above](x) at (0,1) {$m$};
 \node (x) at (1,0) {$v_1\inv$};
\node [above](x) at (1,1) {$v_1$};
\node  (x) at (2,0) {$v_2\inv$}; 
 \node [above] (x) at (2,1) {$v_2$};
\node (x) at (3.5,0) {$x_1\inv$};
 \node [above](x) at (3.5,1) {$x_1$}; 
 \node (x) at (4.5,0) {$x_2\inv$};
\node [above](x) at (4.5,1) {$x_2$}; 
\node (x) at (5.5,0) {$x_3\inv$};
\node [above](x) at (5.5,1) {$x_3$}; 
\node (x) at (6.5,0) {$x_4\inv$};
 \node [above](x) at (6.5,1) {$x_4$}; 
\node [blue] (P) at (1.5,.65) {$P$};
\node [blue] (Pc) at (-.5,.3) {$P^*$};
\node  [red]  (L) at (5,.65) {$lk(P)$};
 \end{tikzpicture}
\end{center}
\caption{Example of a \GW-partition based at $m$ for the graph in Figure~\ref{Gamma}}
\label{GWpicture}
\end{figure}
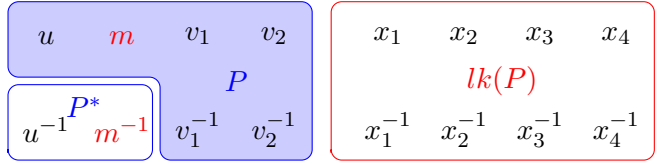

\begin{lem}\label{GWautos}Let $\phi(P,m)$ be a $\Gamma$-Whitehead automorphism. Then  
\begin{enumerate}
 \item $\phi(P,m)\inv=\phi(P\setminus\{m\}\cup \{m\inv\}, m\inv)$
 \item $\phi(P^*,m^{-1})$   is equal to $\phi(P,m)$ composed with conjugation by $m$, 
 so the two are equal as outer automorphisms.  
 \end{enumerate}
 \end{lem}  
 \begin{proof} Clear from the definitions.
 \end{proof}

For a \GW-partition $\WP=\{P|P^*|lk(m)^\pm\}$ based at $m$ we define the outer automorphism $\varphi(\WP,m)$ to be  
                 \[
                 \varphi(\WP,m)=
                 \begin{cases}
                 \text{the image of }\phi(P,m) & \text{if } m\in P\\ 
                 \text{the image of }\phi(P^*,m) & \text{if }  m\in P^*
                 \end{cases}
                 \]
  We will call $\varphi(\WP,m)$ an {\em outer $\Gamma$-Whitehead automorphism}. By  Lemma~\ref{GWautos}, $\varphi(\WP,m)=\varphi(\WP,m\inv)$, so we can think of the $m$ in $\varphi(\WP,m)$ as a vertex of $\Gamma$ instead of an element of $V^\pm$.   
  
  \begin{notation} We extend the relations $\leq,\sim,  \leq_\circ, \sim_\circ, \leq_{\star}, \sim_{\star}$ etc.  to elements of $V^\pm$ by saying a relation holds if and only if it holds for   the corresponding vertices.
  \end{notation}

If $P$ is a \GW-subset  based at $m$, let   $max(P)$ be  the elements $n\in P$ with $n\sim_\circ m$ and $n\inv\not\in P$.  Then $P$ is also  based at any $n\in max(P)$.  Since all elements of $max(P)$ have the same link,  we will  write  $\WP=\{P|P^*| lk(P)\}.$   There is a $\Gamma$-Whitehead automorphism $\phi(P,m)$ for each $m\in max(P)$.

\begin{definition}(\cite{CSV}, Definition 3.3) Let $\WP$ and $\WQ$ be \GW-partitions, with $\WP$ based at $m$ and $\WQ$ based at $n$.  Then $\WP$ and $\WQ$ are {\em compatible} if   either  
\begin{enumerate}
\item $P^{\times}\cap Q^{\times}=\emptyset$ for at least one choice of sides $P^{\times}\in\{P,P^*\}$ and $Q^{\times}\in\{Q,Q^*\},$ or 
\item $[m,n]=1$ but $st(m)\neq st(m)$.  
\end{enumerate}

\end{definition}

\begin{rem}\label{CSVerror} This is the definition of compatibility  given in \cite{CSV}.  However the definition that is actually used in the proofs in that paper is weaker: condition (2) needs to be replaced by 
\begin{itemize}
\item   $[m,n]=1$ but $m\neq n$.  
\end{itemize}
We will call this {\em weak compatibility}.  The proofs in this paper use the stronger notion of compatibility, but we show in Lemma~\ref{weak} that this does not change the results of this paper.  
\end{rem}

If  the bases $m$ of $\WP$ and $n$ of $\WQ$ do not commute, the following lemma constrains the relationships between sides of $\WP$ and $\WQ$. 

\begin{lem}(\cite{CSV}, Lemma 3.4)
\label{useful}
Suppose that $\WP=\{P|P^*| lk(P)\}$ based at $m$ and $\WQ=\{Q|Q^*|lk(Q)\}$ based at $n$ are compatible,  $m$ and $n$  do not commute and $P\cap Q=\emptyset$. Then $P\cap lk(Q)=\emptyset$. In particular, $P\subseteq Q^*$ and $Q\subseteq P^*$.
\end{lem}

\subsection{Outer space $\mathcal O_\Gamma$ and its  spine $K_\Gamma$}\label{spine}

In \cite{CSV} an ``outer space" $\mathcal O_\Gamma$ was defined on which $\UG$ acts properly, and it was proved that $\mathcal O_\Gamma$ is contractible.  The proof proceeds by retracting $\mathcal O_\Gamma$ equivariantly onto a {\em spine}  $K_\Gamma,$ which is the geometric realization of a partially ordered set (poset) of {\em marked $\Gamma$-complexes} $(g,X)$ with $\pi_1(X)\iso \AG$.    

The simplest example of a $\Gamma$-complex is the {\em Salvetti complex} $S_\Gamma$.  This is the   non-positively curved (i.e. locally CAT(0)) cube complex with a single 0-cell,  one edge for each vertex of $\Gamma$, and one $k$-cube for each $k$-clique in $\Gamma$.  A general $\Gamma$-complex  $X$ is a certain type of non-positively curved cube complex which can be collapsed along hyperplanes to  produce the Salvetti complex.  A  {\em marking} is a  homotopy equivalence  $g\colon S_\Gamma\to X$ from a fixed standard Salvetti $S_\Gamma$ whose fundamental group we identify with $\AG,$  with the property that if $c\colon X\to S_\Gamma$ is a sequence of hyperplane collapses then the composition $c\circ g\colon S_\Gamma \to X\to S_\Gamma$    induces an element of $\UG$ on the level of fundamental groups.   The group $\UG$  acts on vertices $(g,X)$ of $K_\Gamma$ by changing the marking.    

Each $\Gamma$-complex $X$  is constructed using  a collection of pairwise-compatible  \GW-partitions (see \cite{CSV} for the construction; we will not need to know the details).  If we start with $X=S$ homeomorphic to $S_\Gamma$ and fix a marking $g\colon S_\Gamma\to S$, the empty collection corresponds to the marked Salvetti  $(g, S)$, and the partially ordered set of all  compatible collections of  \GW-partitions (ordered by inclusion) corresponds precisely to the star of $(g, S)$ in $K_\Gamma$.  In other words, each (ordered) compatible collection $(\WP_1,\ldots,\WP_k)$ corresponds to a $k$-simplex
$$\emptyset\subset \{{\WP_1}\}\subset\{{\WP_1,\WP_2}\}\subset\ldots\subset \{{\WP_1,\ldots,\WP_k}\}$$  
of the star; we abuse notation by writing  
 $$\emptyset\subset \WP_1\subset{\WP_1\WP_2}\subset\ldots\subset {\WP_1\WP_2\cdots\WP_k}.$$   The entire complex $K_\Gamma$ is the orbit of a single such star, so the dimension of $K_\Gamma$ is equal to the maximal size of a compatible collection of  \GW-partitions.   (Lemma~\ref{weak} shows that this size does not depend on whether one uses compatibility or weak compatibility.)

Since $Out(\AG)$ is known to have torsion-free subgroups of finite index,  the fact that $\UG$ acts properly on $K_\Gamma$ gives  
\begin{thm}\label{dimension}
The \vcd\ of $\UG$ is less than or equal to the maximal size of a compatible collection of \iGW-partitions.
\end{thm}
\subsection{Cube complex structure of $K_\Gamma$} Note that any ordering of $\{{\WP_1,\ldots,\WP_k}\}$ gives a $k$-simplex in the star of $(g, S_\Gamma)$, and the union of all of these simplices forms a $k$-dimensional cube (see Figure~\ref{cube}). Thus $K_\Gamma$ in fact has the structure of a cube complex, with one $k$-dimensional cube for each compatible collection $\Pi=\{{\WP_1,\ldots,\WP_k}\}$, which we will denote $c(\emptyset, \Pi)$.    The faces of $c(\emptyset,\Pi)$ correspond to pairs $\Pi_1\subset \Pi_2$ of subsets of $\Pi$; in particular the maximal faces of are of the form $c(\emptyset,\Pi\setminus \{\WP\})$ and 
$c(\{\WP\},\Pi)$ for some $\WP\in \Pi$.   

\begin{figure}
\begin{center}  
\begin{tikzpicture}
\draw (0,0) to (2,0) to (2,2) to (0,2) to(0,0);
\draw (0,0) to (-.7,-.8); 
\draw (2,0) to (1.3,-.8);
\draw (2,2) to (1.3,1.2);
\draw (0,2) to (-.7,1.2);

\draw[densely dotted] (0,0) to (2,2);
\draw[densely dotted] (-.7,-.8) to (2,2);
\draw[densely dotted] (-.7,-.8) to (2,0);
\draw[densely dotted] (-.7,-.8) to (0,2);
\draw[densely dotted] (-.7,-.8) to (2,2);
\draw[densely dotted] (-.7,1.2) to (2,2);
\draw[densely dotted] (1.3,-.8) to (2,2);

\begin{scope} [xshift = -.7cm, yshift= -.8cm];
\draw (0,0) to (2,0) to (2,2) to (0,2) to(0,0);
\draw[densely dotted] (0,0) to (2,2);
\node [below left] (a) at (0,0) {$(g,S)=\emptyset$};
\node [left] (b) at (0,2) {$\WP_2$};
\node [below right] (c) at (2,0) {$\WP_1$};
\node [right] (d) at (2,2) {$\WP_1\WP_2$};
\end{scope}
\node [left] (a) at (0,0) {$\WP_3$};
\node [above] (b) at (0,2) {$\WP_2\WP_3$};
\node [right] (c) at (2,0) {$\WP_1\WP_3$};
\node [above right] (d) at (2,2) {$\WP_1\WP_2\WP_3$};

 \end{tikzpicture}
 \end{center}\caption{The cube $c(\emptyset,\WP_1\WP_2\WP_3)$ in the star of $(g,S)$ in $K_\Gamma$}\label{cube}
 \end{figure}
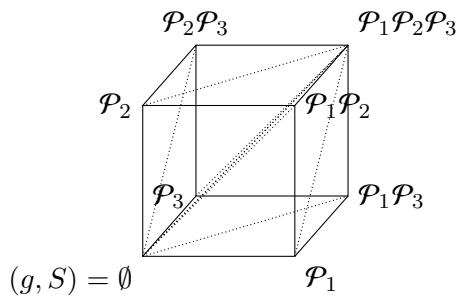

\section{Free abelian subgroups of $\UG$}\label{abelian}

In this section we relate the dimension of $K_\Gamma$ to  abelian subgroups of $Out(A_\Gamma)$ by constructing abelian subgroups freely generated by outer $\Gamma$-Whitehead automorphisms associated to compatible collections of \GW-partitions. We start by determining exactly when  two of these commute.

\subsection{Commuting   $\Gamma$-Whitehead automorphisms}

\begin{definition} Let  $v$ be a vertex of $\Gamma$.  A  \GW-partition $\WP$ {\em splits} $v$ if $v$ and $v^{-1}$ are in different sides of $\WP$.  
\end{definition}

\begin{thm}\label{commute}  Let $\phi(P,m)$ and $\phi(Q,n)$ be $\Gamma$-Whitehead automorphisms.   If $[m,n]=1$ then $\phi(P,m)$ commutes with $\phi(Q,n)$.  If $[m,n]\neq 1$  let $\WP=\{P|P^*| lk(P)\}$ and $\WQ=\{Q|Q^*|lk(Q)\}$ be the associated \iGW-partitions.  Then the outer automorphisms   $\varphi(\WP,m)$ and  $\varphi(\WQ,n)$  commute  if and only if $\WP$ and $\WQ$ are compatible, $\WQ$ does not split $m$ and $\WP$ does not split $n$.
\end{thm}

\begin{proof}
If $m$ and $n$ commute, the automorphisms clearly commute, so we only need to consider the case that $m$ and $n$ do not commute.

Suppose first that $\WP$ and $\WQ$ are compatible.   Replacing $(P,m)$ by $(P^*,m^{-1})$ and/or $(Q,n)$ by $(Q^*,n^{-1})$ if necessary (which does not change $\varphi(\WP,m)$ or $\varphi(\WQ,n)$), then by the definition of compatibility we may assume that $P\cap Q=\emptyset, m\in P$ and $n\in Q.$

  If both $m^{-1}$ and $n^{-1}$ are in $P^*\cap Q^*$, then $\phi=\phi(P,m)$  affects only elements of $P$ and their inverses, and $\psi=\phi(Q,n)$  affects only elements of $Q$ and their inverses. In particular  $\phi$ fixes $n$ and $\psi$ fixes $m$.   If $x\in P$ and  $x\inv\in Q$ then $\phi$ and $\psi$ act on opposite sides of $x$.  It follows that  $\phi\psi(x)=\psi\phi(x)$ for all generators $x$.

 If  $n^{-1}\in P$ and $m^{-1}\in Q$, then    $\phi\psi(m)= mnm$ 
   while $\psi\phi(m)= nm$.  Since these are not conjugate,    $\phi\psi$ and $\psi\phi$ do not differ by an inner automorphism, i.e. they do not commute as outer automorphisms.

  If $n^{-1}\in P^*$ but $m^{-1}\in Q$, then $\phi \psi(n)=n=\psi\phi(n)$ and
   $\phi \psi(m)=nm=\psi\phi(m)$ so we need a different argument to show that $\phi$ and $\psi$ do not commute.  Since $P$ must have at least two elements, there is $v\in P$ with $v\neq m$: 
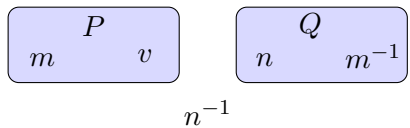
    \begin{center}
    \begin{tikzpicture}[xscale=1.5]
\draw [rounded corners, fill=blue!15] (0,0) to (0,1) to (1.5,1) to (1.5,0)   --cycle; 
\draw [ rounded corners, fill=blue!15] (2,0) to (2,1) to (3.5,1) to (3.5,0)   --cycle;
\node (P) at (.75,.75) {$P$};
\node (m) at (.3,.3) {$m$};
\node (v) at (1.2,.35) {$v$};
\node (im) at (2.25,.3) {$n$};
\node (Q) at (2.65,.75) {$Q$};
\node (n) at (3.2,.35) {$m^{-1}$};
\node (in) at (1.75,-.4) {$n^{-1}$};
\end{tikzpicture}
\end{center}
Since $P\subset Q^*$ by Lemma~\ref{useful}, $v$ does not commute with $m$ or $n$, so $v,m$ and $n$ generate a free group of rank three.  
Since $\phi\psi$ and $\psi\phi$ agree on two generators of this free group, they differ by an inner automorphism if and only if they are equal.    

The effects of $\phi \psi$ and $\psi\phi$ on $v$ are determined by the position of $v^{-1}$: 
\begin{itemize}
\item If $v^{-1}\in P^*\cap Q^*$ then $\phi\psi(v)= vm^{-1}$ and $\psi\phi(v) = vm^{-1}n^{-1}$. 
\item  If $v^{-1}\in Q$ then $\phi\psi(v)= nvm^{-1}$ and $\psi\phi(v) = nvm^{-1}n^{-1}$
\item  If $v^{-1}\in P$ then $\phi\psi(v)= mvm^{-1}$ and $\psi\phi(v) = nmvm^{-1}n^{-1}$
\end{itemize} 
Thus   in all cases,  $\phi\psi$ does not differ from  $\psi\phi$ by an inner automorphism.  
 
This argument applies also to the symmetric case $n^{-1}\in P$ but $m^{-1}\in Q^*$.
  
It remains to consider the possibility that $\WP$ and $\WQ$ are not compatible.  In this case all four quadrants $P\cap Q$, $P\cap Q^*$, $P^*\cap Q$ and $P^*\cap Q^*$ are non-empty.   Using Lemma~\ref{GWautos} we may replace $(P,m)$ by $(P^*,m\inv)$ (which does not change $\varphi(\WP,m)$)  or by  $(P\setminus\{m\}\cup\{m\inv\},m\inv)$  (which replaces $\varphi(\WP,m)$ by its inverse), and similarly replace $(Q,n)$ if necessary,   to obtain one of the following configurations:

\begin{itemize} 
\item If each quadrant contains an element of $\{m,m\inv,n,n\inv\}$, then we may assume $m\in P\cap Q^*, n\in P\cap Q$, $m\inv\in P^*\cap Q$ and $n\inv\in P^*\cap Q^*$. Then  $\phi\psi(n)=nm^{-1}$ and  $\psi\phi(n)=nm^{-1}n^{-1}$ are not conjugate in $\AG$, so $\phi\psi$ and $\psi\phi$ do not differ by an inner automorphism.  

\item If exactly two quadrants contain elements of $\{m,m\inv,n,n\inv\}$, then we may assume $m,n\inv\in P\cap Q^*$ and $n,m\inv\in P^*\cap Q$ so $\phi\psi(m)=nm$ and $\psi\phi(m)=mnm,$  which are not conjugate in $\AG$. 

\item If exactly 3 quadrants contain elements of $\{m,m\inv,n,n\inv\}$ then we may assume $m\in P\cap Q^*, n\in P^*\cap Q$,    $m\inv\in P^*\cap Q^*,$  and either $n\inv\in P^*\cap Q^*$ or $n\inv\in P\cap Q^*$. For either position of $n\inv$ we have $\phi\psi(m)=\psi\phi(m)$ and $\phi\psi(n)=\psi\phi(n)$.   Now $P\cap Q$ does not contain any element of $\{m,m\inv,n,n\inv\}$ but it cannot be empty, so let $v\in P\cap Q$.  Note that $v$ cannot  commute with $m$ or $n$, so $m,n$ and $v$ are the basis of a free subgroup of $\AG$.  Therefore if $\phi\psi$ is conjugate to $\psi\phi$ we must have $\phi\psi(v)=\psi\phi(v)$. A calculation now shows that this is not the case for any position of $v^{-1}$.    
\end{itemize}
 \end{proof}

\begin{cor}\label{conjugation}
If    $\Gamma$-Whitehead automorphisms $\phi(P,m)$ and $\phi(Q,n)$ commute as outer automorphisms, then $\phi(P,m)$ acts on $n$ either trivially or as conjugation by $m$.
\end{cor}
\begin{proof} This is immediate from Theorem~\ref{commute} and the definition of $\phi(P,m)$.
\end{proof}

Let $m\in V^\pm$ and let  $\WP=(P,P^*,lk(P))$ be a \GW-partition based at $m$.     We define the {\em $m$-length of $\WP$} to be   the number of $m$-inseparable subsets in the side  of $\WP$ containing $m$.  
\bigskip
\begin{lem}
\label{length}Let $m\in V^\pm$ and let  $\WP$
  and $\WQ$ be distinct \iGW-partitions based at  $m$, with $m$-length$(\WP)=m$-length$(\WQ)$. Then $\WP$ and $\WQ$ are incompatible.
\end{lem}

\begin{proof}The sides of $\WP$ and $\WQ$ containing $m$  are unions of elements of $\bI(m)$.  If they have the same $m$-length but are different,  then all sides of $\WP$ and $\WQ$ must intersect non-trivially.   
\end{proof}

\begin{lem}\label{nest} Let $m\in V^\pm$ and let $ \WP_1,\ldots,\WP_k$ be pairwise-compatible \iGW-partitions based at  $m$.  Let $P_i$ be the side of $\WP_i$ that contains $m$.  Then after reordering we may assume $P_1\subset P_2\subset\ldots\subset P_k$.  
 \end{lem}
  \begin{proof} For each $i\neq j$, $P_i\cap P_j$ contains $m$, so is not empty, and  $P_i^*\cap P_j^*$ contains $m^{-1}$, so is not empty.    Therefore, by compatibility, either $P_i\cap P_j^*=\emptyset$, which implies $P_i\subset P_j$, or $P_j\cap P_i^*=\emptyset$, which implies $P_j\subset P_i$. Therefore we can renumber  the $P_i$ in order of size to obtain $P_1\subset\cdots\subset P_k.$  
\end{proof}
\begin{figure}
\begin{center}
 \begin{tikzpicture}
 \node (v) at (-1,.7) {$m^{-1}$};
 \node (v) at (-1,.2) {$u$};
\draw [rounded corners] (0,0) to (0,1) to (1.4,1) to (1.4,0)   --cycle; 
\node (P1) at (1,.7) {$P_1$};
\node (m) at (.4,.5) {$m$};
\draw [rounded corners] (-.1,-.1) to (-.1,1.1) to (3,1.1) to (3,-.1)   --cycle; 
\node (Pi1) at (2.5,.7) {$P_{i-1}$};
\node (dots) at (1.8,.5) {$\ldots$};
\draw [rounded corners] (-.2,-.2) to (-.2,1.2) to (3.9,1.2) to (3.9,-.2)   --cycle; 
\node (Pi) at (3.4,.7) {$P_{i}$};
\node (ui) at (3.5,.2) {$u^{-1}$};
\draw [rounded corners] (-.3,-.3) to (-.3,1.3) to (5.8,1.3) to (5.8,-.3)   --cycle; 
\node (Pk) at (5.3,.7) {$P_{k-1}$};
\node (dots) at (4.3,.5) {$\ldots$};
\draw [rounded corners] (-.4,-.4) to (-.4,1.4) to (6.6,1.4) to (6.6,-.4)   --cycle; 
\node (Pk) at (6.2,.7) {$P_{k}$};
\node (v) at (6.2,.1) {$v$};
\end{tikzpicture}
\end{center}
\caption{Proof of Proposition~\ref{independent}}\label{fig:independent}
\end{figure}
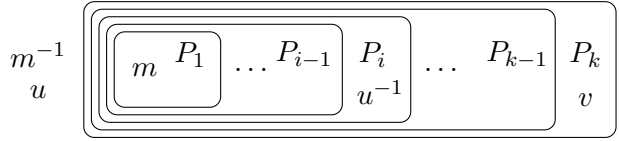

\begin{prop}
\label{independent}
 Let $m\in V^\pm$ and suppose $\WP_1,\ldots,\WP_k$ are pairwise compatible  \iGW-partitions based at $m$.    Then the subgroup of $\UG$ generated by the $ \varphi{(\WP_i,m)}$ is free abelian of rank $k$.
\end{prop}

\begin{proof} 
Let  $P_1\subset \cdots\subset P_k$  be the sides of the $\WP_i$ that contain  $m$ as in  Lemma \ref{nest} (see  Figure~\ref{fig:independent}), and let $\phi_i=\phi(P_i,m).$ Suppose $g=\phi_1^{n_1}\ldots \phi_k^{n_k} $ is inner, and let $u\in P_k^*, u\neq m^{-1}$.  Then
\[
g(u)=\left\{\begin{array}{ll} 
u& \mbox{ if } u^{-1}\in P_k^* \\
m^{a}u & \mbox{ if } u^{-1}\in P_k
\end{array}\right.
\]
where $a=\sum_{\ell=i}^k n_\ell$ if $u^{-1}\in P_i\cap P_{i-1}^*$.   
Since $g(u)$ is conjugate to $u$, we must have $a=0$, i.e. $g(u)=u$ in all cases, so  $g$ is not just inner, but is actually the identity.  
Now let $v\in P_k\cap P_{k-1}^*$.  
Then
\[
g(v)=\left\{\begin{array}{ll} 
vm^{-n_k}& \mbox{ if } v^{-1}\in P_k^*\\
m^{b} m^{n_k}v m^{-n_k}& \mbox{ if } v^{-1}\in P_k
\end{array}\right.
\]
where $b=0$ if $v^{-1}\in P_k\cap P_{k-1}^*$ and  $b=\sum_{\ell=j}^{k-1}n_\ell$ if $v^{-1}\in P_j\cap P_{j-1}^*$ for some $j<k$.  Since $g=id$, this implies $n_k=0$ in all cases.  Repeating this argument with $P_1\subset\ldots\subset P_r$ for each $r<k$ gives $n_r=0$ for all $r$.  
\end{proof}

\begin{prop}\label{dependent} 
Let $\{\WP_1,\ldots,\WP_n\}$ be a maximal compatible collection of    \iGW-partitions based at $m$.  Suppose $\WQ$ is another  \iGW-partition based at $m$.  Then $\varphi{(\WQ,m)}$ is in the subgroup  $G$ of $\UG$ generated by the $\varphi{(\WP_i,m)}$.
\end{prop}  

\begin{proof}  
Let $P_i$ be the side of $\WP_i$ containing $m$.  
By maximality of the collection together with Lemma \ref{length} we know that $\bI(m)$ has exactly $n+1$ elements $U_1,\ldots, U_{n+1}$ other than $\{m\}$ and $\{m\inv\}$ and (after setting $P_0=\{m\}$ and possibly reordering) we have $P_{i}=P_{i-1}\cup U_i$.  Define $P_{n+1}=P_n\cup U_{n+1}$ and set $V_i=U_i\cup \{m\}$.  Then for all $i$ with $1\leq i\leq n+1$ we have 
$\phi(V_i,m)= \phi(P_i,m)\circ\phi(P_{i-1},m)\inv,$ so the corresponding outer automorphism is in $G$.   
  
 Each $m$-inseparable set in the side $Q$ of $\WQ$ containing $m$ is one of the $U_i$, so we have $Q=\{m\}\cup  U_{i_1}\cup\dots\cup U_{i_k}$.  Then
  \[
  \phi(Q,m) = \phi(V_{i_1},m)\circ\cdots\circ\phi(V_{i_k},m),
 \]
so $\varphi(\WQ,m)$ is in $G$.  
\end{proof}

We next show how Propositions~\ref{independent} and \ref{dependent} generalize to the situation where all partitions are based in the same abelian equivalence class.  

\begin{lem}\label{abeliancompatible} Let $\WP=\{P|P^*| lk(P)\}$ be based at $v\in\Gamma$ and let $w\in\Gamma$ be a distinct vertex with  $st(w)=st(v)$.  Let $P$ be the side of $\WP$ containing $v$, set $P_{v,w}= P\setminus\{v\}\cup\{w\}$ and
 $\WP_{v,w}=
                  \left\{P_{v,w}|P^*_{v,w}|lk(w)^\pm\right\}.$
Then
\begin{enumerate} \item  $\WP$ and $\WP_{v,w}$ are compatible. 
\item  If $\WR$ is compatible  with $\WP$ then
  $\WR$ is also compatible with $\WP_{v,w}$
  \item If $\varphi(\WR,s)$ commutes with $\varphi(\WP,v)$ then $\varphi(\WR,s)$ commutes with $\varphi(\WP_{v,w},w)$
  \end{enumerate}
\end{lem}

\begin{proof} 
For the first statement, notice that  $P\cap P^*=\emptyset$ implies $P\cap  ( P_{v,w})^*=\emptyset$ since $w\in lk(v)$.  

Now suppose $\WR$ is based at $s$ and is compatible with $\WP$.  If $[v,s]=1$ and $st(v)\neq st(s)$, then $st(w)\neq st(s)$ so $\WR$ is compatible with $\WP_{v,w}$.  

If $st(v)=st(s)$ or if $[v,s]\neq 1$ then by possibly renaming sides may assume $P\cap R=\emptyset$.  The only  element of $P_{v,m}$ which is not in  $P$   is $w$.
If $st(s)=st(v)=st(w)$ then  $w\in lk(R)$, and  if $[s,v]\neq 1$ then $R\subset P^*,$  which does not contain $w$.  In either case   $w\not\in R$,  so $P_{v,w}\cap R=\emptyset$ and $\WP_{v,w}$ is compatible with $\WR$.  

For the third statement, by Theorem~\ref{commute} it remains to check that if $[w,s]\neq 1$ then $\WP_{v,w}$ doesn't split $s$ and $\WR$ doesn't split $w$.  The first statement clear since $w,w\inv\in lk(P)^\pm$, which doesn't intersect $R$.  The second follows since $\WP$ doesn't split $s$, and the only difference between $\WP$ and $\WP_{v,w}$ is the base $w$.  
 \end{proof}
 
\begin{rem}\label{starless} If $st(v)\subset st(w)$ and $P_{v,w}=P\setminus(\{v\}\cup lk(v)^{\pm})\cup \{w\}$, then statements $(1)$ and $(3)$ of Lemma~\ref{abeliancompatible} hold and statement $(2)$ holds unless $st(s)=st(w)$. 
 \end{rem}
 
 We say that $\WP_{v,w}$ in Lemma~\ref{abeliancompatible} is obtained from $\WP$ by {\em exchanging $v$ for $w$}.  
 
 \begin{cor}\label{swap} Let $\Pi$ be a maximal compatible collection of \iGW-partitions, and let $[v]$ be an abelian equivalence class of $\Gamma$.  If $\WP\in \Pi$ is based at $v\in [v]$, then
 $\Pi$ contains every \iGW-partition that can be obtained from $\WP$ by exchanging $v$ for a different element $w\in [v].$
 \end{cor}
 
 \begin{definition} Let $\WP$ be a \GW-partition based at $m$.  Define $\oring{\WP}$ to be the partition of $V^\pm\setminus st(m)^\pm$ obtained by intersecting each side of $\WP$ with $V^\pm\setminus st(m)^\pm$.
 \end{definition}
 
 \begin{lem}\label{comp} Let $ \WP_1,\ldots,\WP_k$ be pairwise-compatible \iGW-partitions based at $m_i\in[m]$ for some abelian equivalence class $[m]$. Then for some ordering of the $\WP_i$ and some choice of sides $P_i$  we have $\oring P_1\subseteq \oring P_2 \subseteq\ldots\subseteq \oring P_k$.  
 \end{lem}
  \begin{proof}  Let $P_i$ be the side of $\WP_i$ that contains $m_i$, and set $\oring P_i=P_i\setminus \{m_i\}$.    Fix $m\in[m]$ and for each $i$ define $P_{i,m}  = P_i\setminus \{m_i\}\cup \{m\}=\oring P_i\cup\{m\}$.  Then the $P_{i,m} $ are all compatible by Lemma~\ref{abeliancompatible}, and 
 by Lemma~\ref{nest}  we can renumber  the $P_{i,m} $ in order of size to obtain $P_{1,m} \subseteq\cdots\subseteq P_{k,m} $.  Removing $m$ from each $P_i$ now gives  $\oring P_1 \subseteq \oring P_2 \subseteq\ldots\subseteq \oring P_k.$
\end{proof}

 \begin{prop}
\label{independent abelian} Let $[m]$ be an abelian equivalence class and 
suppose $\Pi=\{\WP_1,\ldots,\WP_k\}$ is a compatible collection of distinct  \iGW-partitions based at elements $m_i\in[m]$. 
 Then the subgroup of $\UG$ generated by the $ \varphi{(\WP_i,m_i)}$ is free abelian of rank $k$.
\end{prop}

\begin{proof}Since  $[m]$ is abelian the base $m_i$ of each $\WP_i$ is uniquely determined by $\WP_i$, so we may partition $\Pi$ into subsets $\Pi_n$ with the same base $n\in[m]$.  The subgroup generated by the $\varphi(\WP_i,m_i)\in \Pi_n$   is free abelian by Proposition~\ref{independent}, and the intersection of any two of these is trivial since they use different multipliers.  Therefore the subgroup generated by all of the $\varphi(\WP_i,m_i)$ is the direct product of the subgroups $A_n$ generated by the $\varphi(\WP_i,m_i)\in \Pi_n$, so is free abelian of rank $k$.  
 \end{proof}

 \begin{prop} \label{dependent abelian} Let $\{\WP_1,\ldots,\WP_k\}$ be a maximal compatible collection of    \iGW-partitions based at elements $m_i$ of an abelian equivalence class $[m]$.  Suppose $\WQ$ is another  \iGW-partition based at some $n\in [m]$.  Then $\varphi{(\WQ,n)}$ is in the subgroup generated by the $\varphi{(\WP_i,m_i)}$. 
\end{prop}  
\begin{proof}  Since $\Pi$ is maximal, $n=m_i$ for some $i$ by Lemma~\ref{swap}.  Also, the partitions $\WP_i$ based at $n$ form a maximal collection  of such partitions.  So by Proposition~\ref{dependent} $\varphi(\WQ,n)$ is in the subgroup generated by the $\varphi(\WP_i,n)$.   
\end{proof}

\subsection{Large abelian subgroups of $U(A_\Gamma)$}

\begin{definition} For any subset $U\subset V$ of vertices of $\Gamma$, let $M(U)$ denote the largest possible size of a compatible collection of \GW-partitions, each based at some $u\in U$.  
\end{definition}

 \begin{example}\label{dimKG}   $M(V)=dim(K_\Gamma)$, by Theorem~\ref{dimension}.
 \end{example}
  \begin{example}\label{Mm} $M(m)=|\bI(m)|-3,$ since any \GW-partition based at $m$ gives a thick partition of $\bI(m)$, and the largest compatible set of such partitions is obtained by adding one element of $\bI(m)$ at a time.  
 \end{example}
 
 \begin{notation}  Let $\Pi$ be a compatible collection of \iGW-partitions, and  $U\subset V$ is a subset of vertices of $\Gamma$. Then 
\begin{itemize}
\item $\Pi_U=\{\WP\in\Pi :  \WP$ is based at some $u\in U\}$ and
\item $\Pi^\pm$ is the set of \iGW-subsets of $V^\pm$ which are sides of elements of $\Pi$.
\end{itemize} 
 \end{notation}

In this section we find a free abelian subgroup of $\UG$ of rank $M(L)$, where $L$ is the set of principal vertices of $\Gamma,$ i.e. the set of vertices of $\Gamma$ with maximal links.   This subgroup will be generated by $\Gamma$-Whitehead automorphisms, and we will also show that every abelian subgroup freely generated by $\Gamma$-Whitehead automorphisms has rank at most $M(L)$.    
 The following lemma shows that  this bound is unchanged if we use the weaker notion of compatibility (see Remark~\ref{CSVerror}.)

\begin{lem}\label{weak} Let   $U\subset V$ be any subset of vertices of $\Gamma$, and let $\mu(U)$ denote the largest possible size of a weakly compatible collection of \iGW-partitions, each based at some $u\in U$.   Then  $\mu(U)=M(U)$.
\end{lem}
\begin{proof} Let $\Pi$ be any collection of weakly compatible partitions of size $\mu(U)$.  For each abelian equivalence class $[v]$ choose $m\in [v]$ such that $|\Pi_m|$ is largest.  Remove all $\WP\in \Pi_{[v]}-\Pi_m$ from $\Pi$, then add partitions $\WP_{m,n}$ for each $\WP\in \Pi_m$ and $n\in [v]$ with $n\neq m$.  By Lemma~\ref{abeliancompatible} the resulting collection $\Pi'$ is a (strongly) compatible collection, and since $|\Pi_m|$ was largest we have $|\Pi'|\geq |\Pi|$.  Therefore, $\mu(U)\leq M(U)$. However, any compatible partitions are weakly compatible so $\mu(U)\geq M(U)$ giving equality.
\end{proof}

In  Lemma~\ref{splits} to Proposition~\ref{PQ*}   we fix  a compatible collection $\Pi$ of \GW-partitions.  Recall that a partition {\em splits} a vertex $v$ if $v$ and $v\inv$ are in different sides of the partition.

\begin{lem}\label{splits}   Suppose  $\WP\in \Pi$  is based at $m$ and  $\mathcal R\in \Pi$ is based at $s\not\sim m$.  If $m$ and $s$ do not commute and $\mathcal R$ splits some vertex in $[m]$, then $m<_\circ s$.    In particular, if $m$ is principal then all of $[m]^{\pm}$ is in the same side of $\mathcal R$.
\end{lem}

\begin{proof}   
We are assuming $m\not\sim s$, so if  $m\not<_\circ s$ there is some $v\in lk(m)$ which is not in $lk(s)$. This $v$ is adjacent to every element of $[m]$ so  all of $[m]$ is in the same component of $\Gamma-lk(s)$.  
\end{proof}

 \begin{lem}\label{position*}  Let $m$ be a principal vertex of $\Gamma$,  $\WP_1,\ldots,\WP_k\in \Pi_{[m]_{\star}}$ and let 
$$\emptyset=\oring P_0 \subset \oring P_1\subseteq\ldots\subseteq\oring P_k\subset \oring P_{k+1} =V^\pm\setminus st(m)^\pm,$$  
where $\oring P_1\subseteq\ldots\subseteq\oring P_k$ is  the nest found in Lemma~\ref{comp}. 
  Suppose  $\WQ\in \Pi\setminus\Pi_{[m]_{\star}}$  is based at $n$. If $m$ does not commute with $n$,   then there is a side $Q$ of $\WQ$ with $Q\subseteq \oring P_i\cap \oring P_{i-1}^*$ for some $i$ with $1\leq i\leq k+1$.  \end{lem}

\begin{proof} Since $\WQ$ is compatible with each $\WP_i$ and $m$ does not commute with $n$, Lemma~\ref{useful} implies that for each $i$  there is some choice of side $Q$ of $\WQ$  so that  either $Q\subset P_i$ or $Q\subset P_i^*$.  Since the base $m_i$ of $\WP_i$ is principal, $\WQ$ does not split $m$, by  Lemma~\ref{splits}.  Since $Q\subset P_i$ or $Q \subset P_i^*$, this means  $Q$ cannot contain  either $m_i $ or $m_i\inv$,  so in fact either $Q\subset \oring P_i$ or $Q\subset \oring P_i^*$.  We claim we can  use the same side $Q$ for all $i.$   
Replacing all $P_i$  by  $P_i^*$ if necessary, we may assume $Q\subset\oring P_i$ for at least one $i\leq k$ (this is because the $\oring P_i^*$ also form a chain). 

If $Q\subset\oring P_1$ then $Q\subset\oring P_j$ for all $j$ and we are done.  Otherwise, take the minimal $i$ with $Q\subset\oring P_i$. Since $Q\not\subset\oring P_{i-1}$ we must have $Q^*\subset\oring P_{i-1}$ or $Q^*\subset\oring P_{i-1}^*$ or $Q\subset\oring P_{i-1}^*$.  
  If $Q^*\subset\oring P_{i-1}$ then $Q\supset P_{i-1}^*\supset P_i^*$, contradicting $Q\subset P_i$.  If $Q^*\subset\oring P_{i-1}^*$ then $\oring P_i\supset Q\supset P_{i-1}$ so $\WQ$ splits $m_{i-1}$, contradicting $\WQ\in \Pi-\Pi_{[m]_{\star}}$.  
  So we must have $Q\subset P_{i-1}^*$, i.e. $Q\subset\oring P_{i}\cap\oring P_{i-1}^*$. 
   \end{proof}

The strategy in several upcoming proofs will be to replace some $\WQ\in\Pi$ by a ``better" \GW- partition $\WP$ compatible with everything in $\Pi$ except $\WQ$, where the feature that makes $\WP$ better will depend on the context.  The following proposition gives us our main tool for doing this. The setup  for this proposition is illustrated in Figure~\ref{LemmaPQ}.  

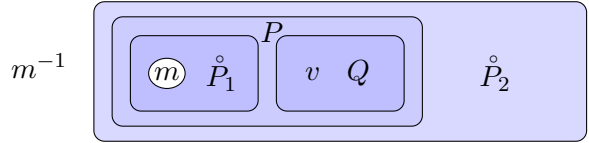
\begin{figure}\begin{center}
\begin{tikzpicture}[xscale=1.2]
\draw [rounded corners, fill=blue!15] (-.4,-.4) to (-.4,1.45) to (5,1.45) to (5,-.4) --cycle;  
\draw [rounded corners, fill=blue!20] (-.2,-.2) to (-.2,1.25) to (3.2,1.25) to (3.2,-.2) --cycle; 
\draw [rounded corners, fill=blue!25] (0,0) to (0,1) to (1.4,1) to (1.4,0)   --cycle; 
\draw [ rounded corners, fill=blue!25] (1.6,0) to (1.6,1) to (3,1) to (3,0)   --cycle;
\node (P1) at (1,.5) {$\oring P_1$};
\draw [fill=white](.4,.5) circle (.2);
\node (m) at (.4,.5) {$m$};
\node (im) at (-1,.6) {$m^{-1}$};
\node (Q) at (2.5,.5) {$Q$};
\node (u) at (2,.5) {$v$};
\node (P2) at (4,.5) {$\oring P_2$};
\node (P) at (1.55,1.05) {$P$};
\end{tikzpicture}
\caption{Proposition~\ref{PQ*}}\label{LemmaPQ}
\end{center}
\end{figure}

\begin{prop}\label{PQ*} Let $m$ be a  principal  vertex  of $\Gamma$,   $\WP_1\in \Pi_m$ and $\WP_2\in \Pi_{[m]_{\star}}$, and choose sides $P_1, P_2$ with $\oring P_1\subset\oring P_2$.  Suppose $u\leq_\circ m$ is contained in $P_2\cap P_1^*$.  Let $Q$ be  a largest subset of $P_2\cap P_1^*$ which is in $\Pi^\pm$ and is based at some $v\sim u$;
 if there are no such subsets, set $Q=\{u\}$.   Let  $\WP$ be the \iGW-partition determined by $P=P_1\cup Q.$ 
 
 If $\WR\in\Pi-\Pi_{[m]_{\star}}$ is not compatible with $\WP$, then some side $R$ of $\WR$ is contained in  $\oring P_2\cap\oring P_1^*$, contains $Q$ and  is based at some $s$ with $s>_\circ u$.\end{prop}

 \begin{proof} Note that $\WP$ is based at $m$.  Since $\WR$ is not compatible with $\WP$ and $s\not\sim_{\star}m$, $s$ and $m$ do not commute.
  
Since $s$ and $m$ do not commute, then by Lemma~\ref{position*} $\mathcal R$ has a side $R$ in $\oring P_{1},\oring P_2\cap\oring P_1^*$ or $\oring P_2^*$.  If  either $R\subset\oring P_1$ or  $R\subset\oring P_2^*$ then  $\WR$ is   compatible with $\WP$, so we must have   $R\subseteq\oring P_2\cap\oring P_1^*$.  Since $\WR$ is compatible with $\WQ$ but not with $\WP$ we must have $R\supset Q$.

 Since $Q$ was of maximal size,  $s\not\in [v]= [u]$.   Thus either $v<_\circ s$ or  there is some $x\in lk(u)\subseteq st(m)$ which is not in $lk(s)$. Such an $x$ would be adjacent to both $v$ and $m$ so  $v$ and $m$ would be  in the same component of $\Gamma-lk(s)$, contradicting the fact that $\WR$ separates $m$ from $v$.
  \end{proof}

 \begin{cor}\label{equiv} Let $\Pi$ be a maximal collection of compatible \iGW-partitions and $[m]$ a non-abelian equivalence class of principal vertices of $\Gamma$.  Then for any $m\in [m]$ the subset $\Pi_{[m]}$ can be replaced by a new set of partitions of the same size to obtain a compatible   collection $\Pi'$   with  $\Pi'_{[m]}=\Pi'_m$.
 \end{cor}
 
 \begin{proof} Fix $m\in [m]=[m]_\circ$   and suppose $\Pi_m=\{\WP_1,\ldots\WP_k\}\neq\emptyset$. Let $P_1\subset\ldots\subset P_k$ be the sides of the $\WP_i$ containing $m$.  
 
 Suppose $\WQ\in \Pi_{[m]}\setminus \Pi_m$ is based at $n\sim m$.  Since $[m]$ is nonabelian, $n$ does not commute with $m$, so it must have a side $Q$ contained in $P_i\cap P_{i-1}^*$ for some $i$.  Take $Q$   maximal with respect to inclusion among all such sides in $P_i\cap P_{i-1}$.    Now take $M$ maximal among all such sides properly contained in $Q$; if there is no such $M$, set $M=\{n\}$.  By Proposition~\ref{PQ*} (applied to $[m]_{\star}=\{m\}$), if some partition $\WR\in \Pi\setminus \Pi_m$ is not compatible with the \GW-partition $\WP$ determined by $P_i\cup M$, then either it is equal to $\WQ$ or it is based at some $s$ with $n<_\circ s$.  i.e. $lk(n)\subsetneq lk(s)$.  But $n$ is principal, so there is no such $s$.  Since $\WQ$ is the only partition in $\Pi$ not compatible with $\WP$, we may replace $\WQ$ by $\WP$ to obtain a new collection of the same size.  
We can continue this process until $\Pi_{[m]}=\Pi_m$.  
 \end{proof}

 \begin{definition} A \GW-partition $\WP$ based at $m$ is {\em principal} if  $m$  is a principal vertex of $\Gamma$.  
 \end{definition}

 \begin{thm}\label{ML subgroup} 
 Let $L$ be the set of principal vertices of $\Gamma$.  Then $\UG$ contains a free abelian subgroup of rank $M(L)$.  
\end{thm}

\begin{proof}    Let $\Pi$ be a maximal compatible collection of principal \GW-partitions, i.e. a collection of size $M(L)$. 

By Corollary~\ref{equiv} we may  assume    $\Pi_{[m]}= \Pi_m$ for all nonabelian equivalence classes $[m]$.   
Using $m$ as multiplier for each $\WP\in \Pi_{[m]}$,  the associated outer $\Gamma$-Whitehead automorphisms  $\varphi(\WP,m)$ pairwise commute. 

If $\WP$ and $\WQ$ in $\Pi$ are based at $m$ and $n$ with $[m,n]=1$ then $\varphi(\WP,m)$ and $\varphi(\WQ,n)$ commute.  

 If $\WP$ and $\WQ$ in $\Pi$ are based at $m$ and $n$ with $[m,n]\neq1$  
  then Lemma~\ref{splits} implies that  $\WP$ does not split $n$ and $\WQ$ does not split $m$, so  $\varphi(\WP,m)$ and $\varphi(\WQ,n)$ commute by Theorem~\ref{commute}.

We now have a collection of pairwise-commuting infinite-order outer automorphisms $\varphi(\WP_i,m_i)$ of size equal to $M(L)$, and we need to show they are independent.  Choose sides $P_i$ for $\WP_i$ containing $m_i$, and set   
 $$\Phi=\phi(P_1,m_1)^{n_1}\ldots\phi(P_k,m_k)^{n_k}.$$ 
We must show that if $\Phi$ is inner then all $n_i=0$.  
 
 Let $\{v_1,\ldots,v_\ell\}$ be the distinct  $m_i$ and define $$\Phi_j= \prod_{m_i=v_j}\phi(P_i,m_i)^{n_i},$$  so $\Phi=\Phi_1\ldots\Phi_\ell$. By  Proposition~\ref{independent} if any of the $\Phi_j$ are inner then the associated $n_i$ are zero; in particular, if $\ell=1$ we are done.  So we may assume no $\Phi_i$ is trivial and $\ell>1$.

If all $v_i$ have the same star, then  we are done by Proposition~\ref{independent abelian}.  Otherwise  without loss of generality we may assume there is $x\in st(v_2)$ with $x\not\in st(v_1)$.   

Replacing $\phi(P_i,m_i)$ by   $\phi(P_i^*,m_i\inv)$  whenever $x\in P_i$ (which doesn't affect their images in $\UG$) we may assume    $\Phi(x)=xU$ for some word $U$ in the $m_i$. Since $\Phi$ is conjugation by some element $W$, this implies $U=1$, so  $W$  is in the centralizer of $x$, which is generated by $st(x)$. Since $v_1\not\in st(x),$  $v_1$ does not appear in any reduced expression for $W.$

Since $\Phi_1$ is not trivial   there is some vertex $y$ with $\Phi_1(y)=v_1^ayv_1^b$, where $a$ and $b$  are not both zero.  If we set  $\Psi= \Phi_2 \cdots\Phi_\ell$ then  $\Phi(y)= \Psi\Phi_1(y)= \Psi(v_1) ^a\Psi(y) \Psi(v_1)^b.$   

By Corollary~\ref{conjugation}, each $\phi(P_i,m_j)$ acts either trivially or as conjugation by $m_i$ on each $m_j$. Thus $\Psi(v_1)$ is conjugate to $v_1$ by a word $U$  in $v_2,\ldots,v_\ell$.  So we have
\begin{align*}
\Phi(y)&= \Psi\Phi_1(y)\\
           &= \Psi(v_1)^a\Psi(y) \Psi(v_1)^b\\
           &= U\inv v_1^aU\Psi(y) U\inv v_1^bU
\end{align*}
We also know that $\Phi(y)=W\inv y W$ for some $W$ that does not contain the letter $v_1.$ But $v_1$ does not commute with $y$  so in order for the powers of $v_1$ in the expression for $\Phi(y)$ above to cancel it must be true that a reduced word representing $\Psi(y)$ does not contain $y$.  In order for this to happen some $\phi(P_i,m_i)$ must have multiplier $m_i=y$. But if $y=m_i$ then $\Psi(y)$ is conjugate to $y$  by Corollary~\ref{conjugation} so the reduced word representing $\Psi(y)$  {\em does} contain $y$, giving a contradiction.  
\end{proof}

\begin{definition} Suppose $ \varphi(\WP_1,m_1), \ldots, \varphi(\WP_k,m_k)$ generate a free abelian subgroup of $\UG$, and let $\Pi=\WP_1,\ldots,\WP_k$.  
Suppose  $[m]$ is  abelian and $\Pi_{[m]}\neq \emptyset$.    Then $\Pi$ is {\em $[m]$-complete} if it contains every \GW-partition $\WQ$ such that
 \begin{itemize}
 \item $\WQ$ is based at  some $n\in[m]$ 
\item  $\varphi(\WQ,n)$ commutes with all $\varphi(\WP_i,m_i)$.  
 \end{itemize}
 \end{definition}

If $\Pi$ is not $[m]$-complete, it can be completed  by adding all possible $\WQ$ satisfying the above conditions.  The base $n$ of any such $\WQ$  is unique since $[m]=[m]_{\star}$, so $\varphi(\WQ,n)$ is determined by $\WQ$.  All of these  $\varphi(\WQ,n)$  can be added to  $\{\varphi(\WP_i,m_i)\}$ to generate an abelian subgroup of possibly larger rank.

{\begin{lem}\label{notcomp}
Suppose  $\varphi(\WP_1,m_1), \ldots, \varphi(\WP_\ell,m_\ell)$ generate a free abelian subgroup  $G$ and  $\Pi=\{\WP_1,\ldots,\WP_\ell\}$ is $[m]$-complete for some abelian $[m]$.  Then $\Pi$ contains a subcollection $\Pi^c$ such that $\Pi^c_{[m]}$ is a compatible collection of \iGW-partitions and the $\varphi(\WP_i,m_i)$ for $\WP_i\in \Pi^c$  generate the same abelian subgroup $G$.  
\end{lem}}

  \begin{proof}  
Let  $\Pi_0$ be maximal compatible subcollection  of   $\Pi_{[m]}$.  If $\WP\in\Pi_0$ is based at $v\in[m]$ then by Lemma~\ref{abeliancompatible} $\WP_{v,w}$ is also in $\Pi_0$ for every $w\in[m]$, since $\Pi$ is $[m]$-complete.    
 
Now consider $\WQ\in \Pi_{[m]}\setminus \Pi_0$,  based at some $n\in [m]$.  By Lemma~\ref{nest}  we may choose sides   $P_i$ of the $\WP\in \Pi_0$ based at $n$ such that 
  $$\{n\}=  P_0 \subset   P_1\subset\ldots\subset P_k\subset  P_{k+1} =V^\pm\setminus lk(n)^\pm\setminus \{n\inv\}.$$

 Take the largest $i\geq 0$ such that the side $Q$ of $\WQ$ containing $n$ also contains $P_i$, and the smallest $j\leq k+1$ such that $Q\subset P_j$.    For each $\ell$ with $i+1\leq\ell\leq j$ let $C_\ell=P_\ell\cap Q$ (see Figure~\ref{fig:notcomp}). 
 
\begin{figure} 
\begin{center}\begin{tikzpicture}
\draw [rounded corners, fill=blue!15] (-.2,-.2) to (-.2,1.2) to (1.6,1.2) to (1.6,.5) to (5.3,.5) to (5.3,-.2)    --cycle; 
\draw [rounded corners] (0,0) to (0,1) to (1.4,1) to (1.4,0)   --cycle; 
\node (P1) at (1,.7) {$P_i$};\node (n) at (.4,.5) {$n$};
 \node (v) at (-1.2,.5) {$n^{-1}$};
\draw [rounded corners] (-.4,-.4) to (-.4,1.4) to (2.5,1.4) to (2.5,-.4)   --cycle; 
\node (Pi1) at (2,.9) {$P_{i+1}$};
\node (ci1) at (2,.2) {$C_{i+1}$};
\draw [rounded corners] (-.6,-.6) to (-.6,1.6) to (4.5,1.6) to (4.5,-.6)   --cycle; 
\node (Pj1) at (4,.9) {$P_{j-1}$};
\node (Pj1) at (4,.2) {$C_{j-1}$};
\node (dots) at (3,.9) {$\cdots$};
\draw [rounded corners] (-.8,-.8) to (-.8,1.8) to (5.5,1.8) to (5.5,-.8)   --cycle; 
\node (Pj1) at (5,.9) {$P_{j}$};
\node (Pj1) at (5,.2) {$C_{j}$};
\end{tikzpicture}
\caption{Proof of Lemma~\ref{notcomp}}\label{fig:notcomp}
\end{center}
\end{figure}
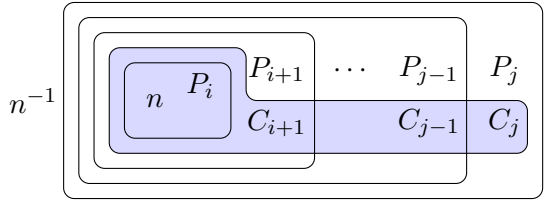

 Then $P^\prime_{\ell}=P_{\ell-1}\cup C_\ell$ is a \GW-subset, and the \GW-partition $\WP^\prime_\ell$ it determines is compatible with all $\WP\in \Pi_0$. Furthermore, since  $\varphi(\WQ,n)$ commutes with all   $\varphi(\WP_j,m_j)$ it follows that $\varphi(\WP^\prime_\ell,n)$ does as well, so   $\WP'_\ell$  must be in  $\Pi_0$ since $\Pi$ is $[m]$-complete and $\Pi_0$ is maximal.    Now 
$$
 \varphi(\WQ,n)=\left(\prod_{\ell=i+2}^j\varphi(\WP_{\ell}^\prime,n)\varphi(\WP_{\ell-1},n)^{-1}\right)\varphi(\WP_{i+1}^\prime,n)
 $$
so  we may eliminate $\WQ$ from $\Pi_{[m]}$ without affecting the abelian subgroup $G$.   Continuing, we   eliminate all partitions in $\Pi_{[m]}$ that are not in $\Pi_0$.  Then $\Pi^c=\Pi\setminus \Pi_{[m]}\cup \Pi_0$ is the required collection.   
  \end{proof}

\begin{thm}\label{largestrank} Any free abelian subgroup of $\UG$ generated by $\Gamma$-Whitehead automorphisms has rank at most $M(L)$. \label{mostML}
\end{thm}
\begin{proof} Suppose $\varphi(\WP_1,m_1),\ldots, \varphi(\WP_k,m_k)$ generate a free abelian subgroup $G$ of rank $r>M(L)$, and let $\Pi=\{\WP_1,\ldots,\WP_k\}$.    If $[m_i,m_j]\neq 1$ then $\WP_i$ is compatible with $\WP_j$  by Theorem~\ref{commute}.  If $[m_i,m_j]=1$ but $m_i\not \sim_{\star}m_j$  then $\WP_i$ is compatible with $\WP_j$ by the definition of compatibility.   So the only incompatible pairs in $\Pi$ live in the same $\Pi_{[m]}$ for some $m$ with $[m]=[m]_{\star}$.   

Fix such an $m$ and add all necessary partitions to $\Pi_{[m]}$ so that $\Pi$ is $[m]$-complete. The corresponding free abelian group contains $G$ as a subgroup. By Lemma~\ref{notcomp} there is a subcollection $\Pi^c$ of $\Pi$ such that $\Pi^c_{[m]}$ is a compatible collection and the corresponding group generated is the same, i.e. it still contains $G$ as a subgroup.  After repeating this for each equivalence class with $[m]=[m]_{\star}$  we may assume that $\Pi$ is a compatible collection.  

Now choose the $\varphi(\WP_i,m_i)$ so that $\Pi$ has the smallest possible number of non-principal partitions for a such a collection.  By Corollary~\ref{equiv} we may assume $\Pi_{[m]}=\Pi_m$ for each principal nonabelian equivalence class  $[m]$ and a choice of representative $m$. Let $\WQ\in \Pi$ be a non-principal partition, based at a vertex $u$ which has maximal link among the non-principal bases.   Since $u$ is non-principal,  $u<_\circ m$ for some $m\in L$. Let $\Pi_{[m]}=\{\WP_1,\dots,\WP_\ell\}$ ($\ell\geq 0$) and choose sides  $P_i$  of $\WP_i$ so that the $\oring P_i$ are nested. By Lemma~\ref{position*}, there is a side $Q$ of $\WQ$ such that $Q\subset\oring P_i\cap\oring P_{i-1}^*$ for some $i\leq \ell+1$. Maximise $Q$ with respect to inclusion over all non-principal partitions in $\Pi_{[u]}$ with sides in $\oring P_i\cap\oring P_{i-1}^*$.   By Proposition~\ref{PQ*} if a partition $\mathcal{R}$ based at $s\not\sim m$ is incompatible with $P_{i-1}\cup Q$, then $\WR$ has a side $R\subset\oring P_i\cap\oring P_{i-1}^*$ with $R\supset Q$ and $u<_\circ s$.   Maximality of $u$ tells us that $\mathcal{R}$ must in fact be a principal partition, so  by replacing $m$ with $s$ and repeating the above arguments we will reach a point where no such incompatible $\mathcal{R}$ exists.  At this point we claim that $P_{i-1}\cup Q=P_i$:  if  $P_{i-1}\cup Q$ was not  in $\Pi$  we could replace $\WQ$ with the partition determined by $P_i\cup Q$ to arrive at a collection of the same size $k$ but one fewer non-principal partition. 

Since $\varphi(\WP_i,m)$ commutes with $\varphi(\WQ,n)$,   $P_i=P_{i-1}\cup Q$ must contain both $u$ and $u^{-1}$, i.e. $u^{-1}\in P_{i-1}$. This implies that $\WP_{i-1}$ splits $u$, contradicting the commutativity conditions of Theorem~\ref{commute}.
 \end{proof}

\begin{cor} Let $G$ be an abelian subgroup of $\UG$ of rank $M(L),$ freely generated by $\{\varphi(\WP_i,m_i)\}$ with $m_i$ principal.  Suppose $m_i$ is not maximal for some $i$, say $m_i<_\star w$.  Then $\varphi(\WP,w)\in G$, where $\WP$ is the partition obtained from $\WP_i$ by exchanging $m_i$ for $w$, as defined in Remark~\ref{starless}.
\end{cor}

\begin{proof} Let $\Pi=\{\WP_i\}$,  and let $\Pi^\prime$ be the $[w]$-completion of $\Pi$. Remark~\ref{starless} shows  that  $$\WP=\WP_i\setminus(\{m_i\}\cup lk(m_i)^\pm)\cup  \{w\}$$   is contained in $\Pi'$.   If $G'$ is the corresponding free abelian subgroup then $G\leq G'$. Theorem~\ref{largestrank} tells us that $rank(G)=rank(G')$ so $G=G'$, thus $\varphi(\WP,w)\in G$.
\end{proof}

\section{Virtual cohomological dimension}\label{VCD}

By Theorem~\ref{dimension} we know $dim(K_\Gamma)=M(V)$ is an upper bound on the \vcd\  of $\UG$ and $M(L)$ is a lower bound  by Theorem~\ref{ML subgroup}.  In this section we give conditions under which $M(V)=M(L)$.

\begin{lem}
\label{far apart}
If   non-equivalent vertices $u,v\in V$ have $d_\Gamma(u,v)\neq 2$ (where $d_\Gamma$ is the length of a shortest path in $\Gamma$), then any partition based at $u$ is compatible with any partition based at $v$. In particular,
\[M(u,v)=M(u)+M(v).\]
\end{lem}
\begin{proof} If $d_\Gamma(u,v)=1$ then $u$ and $v$ commute.  Since we are assuming $u\not\sim v$, the partitions are compatible.  If $d_\Gamma(u,v)\geq 3,$  let $C_v$ denote the element of $\bI(u)$ containing $v$ (and $v^{-1}$) and $C_u$ the element of $\bI(v)$ containing $u$.  Then $C_v$ contains all elements of $\bI(v)$ other than  $C_u$ and $C_u$ contains all elements of $\bI(u)$ other than $C_v$.   This implies that any  thick partition of $\bI(v)$ separating $v$ from $v^{-1}$ is compatible with any  thick  partition of $\bI(u)$ separating $u$ from $u^{-1}$.  
\end{proof}

\begin{thm}\label{condition}
Let $\Gamma$ be a graph  and $L\subseteq V$ its set of principal vertices. Suppose that  every $u\in V\setminus L$  satisfies
\begin{equation}
\hbox{All principal maximal $m$ with $m\geq_\circ u$ are in the same component of $\Gamma-lk(u)$}.
\end{equation}
Then $M(V)=M(L)$. 
\end{thm}

 \begin{proof} Let $\Pi$ be a maximal pairwise-compatible collection of \GW-partitions with $M(V)$ elements.  We will produce a new collection of the same size in which all partitions are principal. By Corollary~\ref{equiv} we may assume $\Pi_{[m]}=\Pi_m$ for each principal nonabelian equivalence class and a choice of representative $m$.

Let $u$ be maximal among all  $u\in V\setminus L$ with $\Pi_u\neq\emptyset$.   By Lemma~\ref{comp}, for any principal $m\geq_\circ u$  we have a nest $$\oring P_0=\emptyset\subset\oring P_1\subseteq \ldots\subseteq\oring P_k\subset \emptyset^*=\oring P_{k+1}, $$
where $\WP_1,\ldots,\WP_k  \, (k\geq 0)$ are the elements of $\Pi_{[m]}$ and $P_i$ is a side of $\WP_i$. By Lemma~\ref{position*}, each $\WQ\in \Pi_{[u]}$ has a side $Q^\times \subset\oring P_{i}\cap\oring P_{i-1}^*$ for some $i$.  For each $m$, let $i(m)$ be the smallest index such that $\oring P_{i(m)}$ contains one of these $Q^\times$.  Choose $m$ such that $|P_{i(m)}|$ is minimal.  Among the  $ \WQ\in \Pi_u$ with $Q^\times\subset\oring P_{i(m)}$ choose one with $Q=Q^\times$ maximal.  

The union $P_{i-1}\cup Q$ determines a \GW-partition  $\WP$  based at $m$. If  there is some $\WR\in\Pi$   not compatible with $\WP$ then   by Proposition~\ref{PQ*} $\WR$ has a side $R\subset\oring P_i$ containing $Q$  and disjoint from $\oring P_{i-1}$, and  $\WR$ is based at some $s>_\circ u$.  By our choice of $u$ this implies that $s$ is principal,  so either $Q$ or $Q^*$ is somewhere in the nest 
$\emptyset\subset\oring R_1\subset \ldots\subset\oring R_\ell\subset \emptyset^*$ associated with $s$.   Since  $s>_\circ u,$ $Q$ does not contain $s$ (if it did, it would split $s$ and we would have $s\leq_\circ u$). Therefore $Q$ is in the nest.  Since $R=R_j$ for some $j$, we have $Q\subset R_j \subsetneq P_i$, contradicting minimality of $|P_i|$.

Now take a proper subset $M\subsetneq Q$ in $\Pi_{u}^\pm$ of maximal size.  If there is no such $M$, take $M=\{u\}$. Proposition~\ref{PQ*} applied to $\Pi\setminus \WQ$ shows that if $\WR$ is not compatible with the \GW-partition $\WP^{\prime}$ determined by $P_{i-1}\cup M$ then either $\WR=\WQ$ or $\WR$ has a side $R\subset\oring P_i$ containing $M$ and disjoint from $\oring P_{i-1}$, and $\WR$ is based at some $s>_\circ u$.  By our choice of $u$, $s$ is principal.  But $s$ and $m$ are on different sides of $\WQ$, contradicting our hypothesis that all principal $v>_\circ u$ are in the same component of $\Gamma-lk(u)$.  
 We can now replace $\WQ$ by $\WP^{\prime}$ to get a new collection of the same size, with one fewer non-principal partition. Continuing, we can replace all non-principal partitions by principal partitions, showing $M(V)=M(L)$.   
 \end{proof}

The following is a special case of Theorem~\ref{condition} which is often very easy to check.  
\begin{cor} If every  non-principal equivalence class of vertices in $\Gamma$ is $<_\circ$ at most one  principal equivalence class, then $M(V)=M(L)$.  
\end{cor}

\section{Examples}\label{sec:examples}

In this section we give a few examples illustrating  both the utility and the limits of Theorem~\ref{condition}.  

\begin{example}
Let $\Gamma$ be the graph with $n$ vertices and no edges, i.e. $A_\Gamma=\mathbb{F}_n$. Here there are no twists so $\UG=Out(\AG)$.  Since all vertices are maximal and equivalent,  Corollary \ref{equiv} implies $M(V)=M(m)$ for any choice of vertex $m$.   Since  $M(m)=2n-3$ (see Example~\ref{Mm}),  this gives (the correct) lower bound of $2n-3$ for the \vcd\ of $Out(F_n)$.  \end{example}

\bigskip

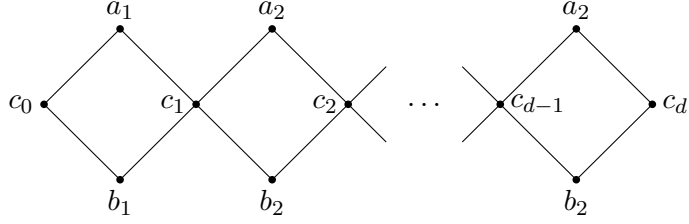
\begin{figure}\begin{center}
\begin{tikzpicture}[scale=1] 
\fill [black] (0,0) circle (.05); \node [left](c0) at (0,0) {$c_0$};
\fill [black] (1,1) circle (.05);\node [above] (a1) at (1,1) {$a_1$};
\fill [black] (1,-1) circle (.05);\node [below] (b1) at (1,-1) {$b_1$};
\fill [black] (2,0) circle (.05); \node [left](c1) at (2,0) {$c_1$};
\fill [black] (3,1) circle (.05);\node [above] (a1) at (3,1) {$a_2$};
\fill [black] (3,-1) circle (.05);\node [below] (b1) at (3,-1) {$b_2$};
\fill [black] (4,0) circle (.05); \node [left](c1) at (4,0) {$c_2$};

\fill [black] (6,0) circle (.05); \node [right](c1) at (6,0) {$c_{d-1}$};
\fill [black] (7,1) circle (.05);\node [above] (a1) at (7,1) {$a_2$};
\fill [black] (7,-1) circle (.05);\node [below] (b1) at (7,-1) {$b_2$};
\fill [black] (8,0) circle (.05); \node [right](c1) at (8,0) {$c_d$};

\node (dots) at (5,0) {$\ldots$};
\draw (0,0) to (1,1) to (2,0) to (3,1) to (4,0) to (3,-1) to (2,0) to (1,-1) to (0,0);
\draw (4.5,.5) to (4,0) to (4.5,-.5);
\draw (5.5,-.5) to (7,1) to (8,0) to (7,-1) to (5.5,.5);
\end{tikzpicture}
\caption{String of diamonds}\label{diamonds}
\end{center}
\end{figure}

\begin{example}
Let $\Gamma$ be a string of $d$ diamonds, as shown in Figure~\ref{diamonds}. 
 Again there are no twists, so $\UG=Out(\AG)$.  The only non-principal vertices are $c_0$ and $c_d$ and there are no \GW-partitions based at either of these, so $M(V)=M(L)$.   Let $\Pi$ be a collection of  size $M(V)$.   We have $[a_i]=\{a_i,b_i\}$ for each $1\leq i\leq d$ so by Corollary~\ref{equiv} we may assume $|\Pi_{\{a_i,b_i\}}|=|\Pi_{a_i}|$ for each $i$. We have $M(a_i)=3$ if $2\leq i\leq d-1$, $M(a_d)=M(a_1)=2$, $M(c_i)=1$ if $2\leq i\leq d-2$ and $M(c_1)=M(c_{d-1})=2$. Therefore,
\[M(V)\leq 3(d-2)+4+d-3+4=4d-1.\]
It is easy to find a collection of  \GW-subsets with $4d-1$ elements (one is given explicitly  in \cite{CSV}), so in fact $M(V)=4d-1$ and  $dim(K_\Gamma)=M(V) =$ \vcd$(Out(\AG))$.

\end{example}

\begin{example}\label{exampletwo}   
Let $\Gamma$ be the graph in Figure~\ref{fork}.   Since $\Gamma$ is a tree,  the \vcd\ of $Out(\AG)$ is equal to $e+2\ell-3=7+8-3=12$ \cite{BCV}.  The only twists are given by the leaf transvections.   These form a normal free abelian subgroup of rank 4 (the number of leaves), with quotient $\UG$, so it is natural to expect that the \vcd\ of  $\UG$ is $8$.

There are no \GW-partitions based at any of the $b_i$ since $\Gamma\setminus st(b_i)$ has only one component.  Any partition based at $v_1$ is compatible with any partition based at a different vertex by Lemma \ref{far apart}, since $a_i,v_0\in st(v_1)$ for each $i$. We have $M(v_1)=|\bI(v_1)|-3=5$.  Now consider partitions based at $a_1, a_2$ or $a_3$.    Choose any one such partition $\WP$, say based $a_1$.  Then for each $a_i$  there are at most two choices of partition compatible with $\WP$ since the side of $\WP$ not containing $a_i$ must be disjoint from the side of $\WQ$ not containing $a_1$.  Say a choice $\WQ$ is based at $m$, then by repeating this argument on disjoint sides there is at most one choice of partition compatible with both $\WP$ and $\WQ$, so $M(a_1,a_2,a_3)=3$ and the largest possible number of \GW-partitions based at principal vertices is $5+3=8$.  Since $M(v_0)=2$, we have $M(V)\leq 10$. In fact equality holds since the following  list of five \GW-subsets determines a compatible collection of distinct \GW-partitions based in $\{a_1,a_2,a_3,v_0\}^{\pm}$:
\[\{a_1,v_0\},\{v_0,a_1,a_1^{-1},b_1,b_1^{-1}\},\{a_2,v_0,a_1,a_1^{-1},b_1,b_1^{-1}\},\{v_0^{-1},a_3,a_3^{-1},b_3,b_3^{-1}\},\{a_3^{-1},v_0\}.\]
 
Thus $M(L)=8\leq \hbox{\vcd}(\UG) \leq dim K_\Gamma=10$. 
    \end{example}
    
    \begin{figure}
\begin{center}
\begin{tikzpicture}[scale=1] 
\fill [black] (0,0) circle (.05); \node [above](c0) at (0,0) {$v_0$};
\fill [black] (1,0) circle (.05);\node [above] (a1) at (1,0) {$v_1$};
\fill [black] (2,1) circle (.05);\node [above] (b1) at (2,1) {$a_1$};
\fill [black] (3,1) circle (.05); \node [above](c1) at (3,1) {$b_1$};
\fill [black] (2,0) circle (.05);\node [above] (a1) at (2,0) {$a_2$};
\fill [black] (3,0) circle (.05);\node [above] (b1) at (3,0) {$b_2$};
\fill [black] (2,-1) circle (.05); \node [above](c1) at (2,-1) {$a_3$};
\fill [black] (3,-1) circle (.05); \node [above](c1) at (3,-1)  {$b_3$};
 \draw (0,0) to (1,0) to (2,0) to (3,0) ;
  \draw   (1,0) to (2,1) to (3,1) ;
    \draw   (1,0) to (2,-1) to (3,-1) ;
\end{tikzpicture}
\caption{$\Gamma$ is a tree with $M(V)=10$, but there is no free abelian subgroup of rank 10 generated by compatible collections of Whitehead automorphisms.}\label{fork}
\end{center}
\end{figure}
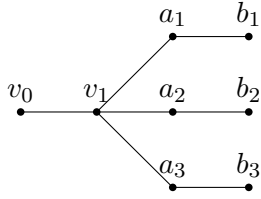
    
        \begin{example}\label{simpler}
    A similar but slightly simpler example is when $\Gamma$ is the tree in Figure~\ref{simpletree}. A quick check yields $M(V)=M(v)+M(u,a_1,a_2)$ and $M(v)=3$. Furthermore, arguing in the same fashion tells us that $M(u,a_1,a_2)\leq 3$, with a possible $\Pi_{\{u,a_1,a_2\}}$ being the \GW-partitions determined by:
    \[\{a_1,u\},\{u,a_1,a_1^{-1},b_1,b_1^{-1}\},\{a_2^{-1},u^{-1}\}.\]
    Thus, $M(V)=6$ so $\dim(K_\Gamma)=6$ but we only find a subgroup $\mathbb{Z}^5\leq \UG$. In the following section it is shown that this particular $\Gamma$ has $VCD(\UG)=5$.

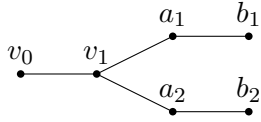
\begin{figure}
\begin{center}
\begin{tikzpicture}[scale=1] 
\fill [black] (0,0) circle (.05); \node [above](c0) at (0,0) {$v_0$};
\fill [black] (1,0) circle (.05);\node [above] (a1) at (1,0) {$v_1$};
\fill [black] (2,.5) circle (.05);\node [above] (b1) at (2,.5) {$a_1$};
\fill [black] (3,.5) circle (.05); \node [above](c1) at (3,.5) {$b_1$};
\fill [black] (2,-.5) circle (.05);\node [above] (a1) at (2,-.5) {$a_2$};
\fill [black] (3,-.5) circle (.05);\node [above] (b1) at (3,-.5) {$b_2$};
 \draw (0,0) to (1,0) to (2,.5) to (3,.5) ;
  \draw   (1,0) to (2,-.5) to (3,-.5) ;
\end{tikzpicture}
\caption{For this tree $M(L)=5$ but dim($K_\Gamma)=6$.}\label{simpletree}
\end{center}
\end{figure}

    \end{example}

\section{Reducing the dimension of $K_\Gamma$}\label{reduction}

In this section we show that, in some cases with $M(V)> M(L)$, we can find an invariant contractible subcomplex of $K_\Gamma$ of smaller dimension. We use the weak notion of compatibility throughout since that is what is actually used in \cite{CSV} to define and prove contractibility of $K_\Gamma$.  

\begin{definition}\label{barbed}
A graph $\Gamma$ is {\em barbed} if  for all non-principal vertices $u$, $d_\Gamma(u,v)=2$ implies $u<_\circ v$.
\end{definition}

\begin{lem} If $\Gamma$ is barbed then every non-principal equivalence class is minimal and has only one element. Furthermore any \iGW-partition based at a non-principal element splits {\em only} that element. 
\end{lem}
\begin{proof} This is immediate.
\end{proof}

All of the graphs in Section~\ref{sec:examples} are barbed. Examples~\ref{exampletwo} and \ref{simpler} are examples of barbed graphs with  $M(V)>M(L)$.  We claim that if $\Gamma$ is barbed and $M(V)>M(L)$, then $K_\Gamma$ equivariantly deformation retracts to a smaller-dimensional complex.  Specifically, every cube in $K_\Gamma$ of dimension $M(V)$ has a free face, and the set of these free faces is invariant under the action of $\UG$.

In Lemmas~\ref{inside} to \ref{innermost} we fix a collection $\Pi$ of pairwise weakly compatible \GW-partitions with $M(V)>M(L)$ elements. Recall that $\Pi^\pm$ denotes the collection of all sides of elements of $\Pi$.   

\begin{lem}\label{inside}
Let $Q\in\Pi^\pm$ be  a non-principal $\mathit\Gamma\!$W  subset, based at some $u\in Q$. If $Q$ contains  some $m\geq_\circ u$ other than $u$, then $Q$ properly contains some $N\in \Pi^\pm$ with  $u\in N$.
\end{lem}

\begin{proof}  Suppose the lemma is false, i.e. no $N\in\Pi^\pm$ is properly contained in $Q$ and also contains   $u$.  

If there are no elements at all of $\Pi^\pm$ properly contained in $Q$, then the \GW-partition determined by  $\{u,m\}$ is (weakly) compatible with all elements of  $\Pi$, contradicting maximality of $\Pi$.

Now take a largest $P\in \Pi^\pm$ properly contained in $Q$, based at some   $n\in P$.  If   $n\not\geq_\circ u$ then there is some $v\not\in lk(n)$ with $v\in lk(u)\subset lk(m)$, so $u,v$ and $m$    are all in the same component of $\Gamma-lk(n)$, so $u, v, m$ and their inverses are all on the same side of $P$, i.e. all are outside $P$.  If this is true for all largest 
$P$ contained in $Q$ we can add the partition determined by $\{u,m\}$ to $\Pi$, again contradicting maximality of $\Pi$.

If some largest $P$ is based at a vertex $n\geq_\circ u$ then $P\cup \{u\}$ is a \GW-subset and the corresponding \GW-partition is (weakly) compatible with all elements of $\Pi$, once again contradicting maximality of $\Pi$.  
 \end{proof}
 
\begin{definition} A \GW-partition $\WQ\in\Pi$   is {\em irreplaceable in $\Pi$} if  $\WQ$ is the only \GW-partition compatible with all elements of   $\Pi\setminus \WQ$. 
\end{definition}

\begin{definition}
A \GW-partition $\WQ\in \Pi$ based at $u$ is \textit{sandwiched in $\Pi$} if there are principal $m\in Q$ and $n\in Q^*$ with $m,n>_\circ u$  such that both   $Q_m=Q\setminus(\{m\}\cup lk(m)^\pm)$ and $Q^*_n=Q^*\setminus (\{n\}\cup lk(n))^\pm$ are in $\Pi^\pm$.
\end{definition}
 
\begin{figure}
\begin{center} 
\begin{tikzpicture}[scale=1] 
\draw [rounded corners, blue,](.5,1.7) to (3.5,1.7) to (3.5,-.4) to (1.5,-.4) to (1.5,.7) to (.5,.7) --cycle;
\draw [rounded corners, blue](-.5,-.4) to (-.5,.6) to (1.4,.6) to (1.4,-.4) to   (-.5,-.4) --cycle;
\draw [rounded corners, blue](.6,.8) to (.6,1.6) to (2.4,1.6) to (2.4,.8) to   (.6,.8) --cycle;
\draw [rounded corners, red](3.6,-.4) to (3.6,1.7) to (4.5,1.7) to (4.5,-.4) to   (3.6,-.4) --cycle;
\node [above](x) at (-1,1) {$ b_2$};
\node (x) at (-1,0) {$ b_2\inv$};
\node  (x) at (0,0) {$a_2\inv$};
\node [above](x) at (0,1) {$a_2$};
 \node (x) at (1,0) {$v_0\inv$};
\node [above](x) at (1,1) {$v_0$};
\node  (x) at (2,0) {$a_1\inv$}; 
 \node [above] (x) at (2,1) {$a_1$};
\node (x) at (3,0) {$b_1\inv$};
 \node [above](x) at (3,1) {$b_1$}; 
 \node (x) at (4,0) {$v_1\inv$};
\node [above](x) at (4,1) {$v_1$}; 
\node [blue] (P) at (2.5,.55) {$Q$};
\node [blue] (P) at (1.5, 1.2) {$P_1$};
\node [blue] (Pc) at (.5,.3) {$P_2$};
\node  [red]  (L) at (4.1,.55) {$lk(Q)$};
 \end{tikzpicture}
\end{center}
\caption{ $Q$ determines a \GW-partition for the tree in Figure~\ref{simpletree} that is sandwiched, with  $Q_{a_1\inv}=P_1$ and $Q^*_{a_2}=P_2$}\label{sandwich}
\end{figure}
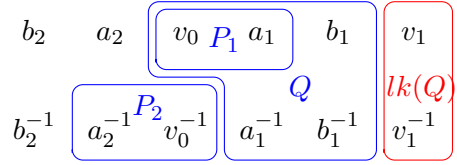

 \begin{lem}\label{irreplaceable}    If a non-principal partition $\WQ\in\Pi$ is sandwiched in $\Pi$  then $\WQ$ is irreplaceable in $\Pi$.
 \end{lem}
\begin{proof} If $Q_m$ and $Q_n^*$ are both in $\Pi^\pm$, then any  replacement for $\WQ$ cannot have a side contained in $Q_m$ or $Q_n^*$ (by maximality of $\Pi$) and cannot split both $m$ and $n$ (since $m$ and $n$ are on different sides of $\WQ$ so are not equivalent). Since $\WQ$ is the only \GW-partition that satisfies these conditions, $\WQ$ is irreplaceable.  
 \end{proof}
 
  \begin{lem}\label{innermost} Let  $\Gamma$ be a barbed graph and $Q\in\Pi^\pm$ innermost among non-principal sides, based at some $u\in Q$.   If $\WQ$ is not sandwiched in $\Pi$, then  $\WQ$ is replaceable by a principal partition.   
 \end{lem}
 
 \begin{proof}
 Since $\Gamma$ is barbed, there are principal elements bigger than $u$ on both sides of $\WQ$.  By Lemma~\ref{inside} there is a proper subset $M$ of $Q$ that is in $\Pi^\pm$ and contains $u$; take a largest such $M$.  Since $Q$ is an innermost non-principal subset, $M$ must be principal based at some $m$, which must be $>_\circ u$ since $M$ separates $u$ from $u\inv$.  Unless $M=Q_{m\inv}=Q\setminus \{m\inv\}\setminus lk(m)^\pm$,   the set 
  $Q^*\cup M\setminus lk(m)^\pm$ has at least two elements on each side so determines a principal \GW-partition which can replace $\WQ$.
   
If   $M=Q_{m\inv}=Q\setminus \{m\inv\}\setminus lk(m)^\pm$, we consider the other side $Q^*$ of $\WQ$.  By Lemma~\ref{inside} there is also a \GW-subset $N\subsetneq Q^*$ with $u\inv\in N$. 
Take a maximal such $N$.   If $N$ is based at $u\inv$ then 
$N\cup Q_{m\inv}$ is principal, based at $m$, and can replace $\WQ$.   So suppose $N$ is based at $n\neq u\inv$.  Since $N$ splits $u$ we must have $u\leq_\circ n$, and since $\Gamma$ is barbed $[u]=\{u\}$ so in fact $n>_\circ u$ and $n$ must be maximal.  If $N=Q^*_{n\inv}=Q^*\setminus\{ n\inv\}\setminus lk(n)^\pm$,  then $\WQ$ is sandwiched, contradicting our assumption.  Therefore the set $Q\cup N\setminus lk(n)^\pm$ is a  \GW-subset and the corresponding principal \GW-partition can replace $\WQ$.   
 \end{proof}

\begin{thm}  Let $\Gamma$ be a barbed graph with $M(V)>M(L)$.  Then the dimension of $K_\Gamma$ is strictly larger than  \vcd$(\UG)$.
\end{thm}

\begin{proof}  Let $\Pi$  be a maximal collection of weakly compatible \GW-partitions with $M(V)>M(L)$ elements.  Then $\Pi$ determines a cube $c(\emptyset,\Pi)$  in $K_\Gamma$ of dimension $M(V)$.  We will find a free face of this cube, namely  $c(\emptyset,\Pi\setminus \WQ)$ for some non-principal $\WQ$ and use it to collapse the cube.  We can do this equivariantly for all such cubes in all of $K_\Gamma$, thereby reducing the dimension of $K_\Gamma$ by $1$.

The cube  $c(\emptyset, \Pi\setminus\WQ)$ is a free face of $c(\emptyset, \Pi)$ if and only if   $\WQ$ is  irreplaceable.   So we are looking for an irreplaceable $\WQ$ in $\Pi$.  
Let $R\in \Pi^\pm$ be an innermost non-principal \GW-subset.  If the corresponding \GW-partition $\WR$ is sandwiched, then it is irreplaceable, by Lemma~\ref{irreplaceable} so we may take $\WQ=\WR$.  If it is not sandwiched, then it can be replaced by a principal \GW-partition $\WP$, by Lemma~\ref{innermost}, to form a new maximal collection $\Pi^\prime$. This new collection has the same size, so must still contain a non-principal \GW-partition.  

{\bf Claim}. If a non-principal $\cS\in \Pi'$  based at $v$  is sandwiched between $S_{m}$ and $S^*_{n}$ in $\Pi^\prime$, then it was already sandwiched in $\Pi$, so is irreplaceable in $\Pi$ by Lemma~\ref{irreplaceable}. 

 {\bf Proof of claim.} If $\cS$ is sandwiched in $\Pi^\prime$ but not in $\Pi$ then either $S_{m}$ or $S^*_{n}$ must be equal to the \GW-subset  we used  in Lemma~\ref{innermost} to replace $Q$.  In all cases this has a side of the form $S=T\cup M$ where $T$ is non-principal based at $u$ and $M$ is principal with  $u\in M$.  It follows that   $M$ is based at $m$ (if $S=S_m$) and $u,v<_\circ m$  (or at $n$ if $S=S^*_n$ and $u,v<_\circ n$) and that $u\neq v$.    But then $M$ splits both $u$ and $v$, which cannot happen in a barbed graph.  \qed

Now let $S$ be an innermost non-principal side in $\Pi^{\prime\pm}$.   If  $\cS$ is sandwiched in $\Pi^\prime$  then by the claim it was already sandwiched in $\Pi$, so is irreplaceable in $\Pi$ and we may take $\WQ=\cS$.   If it is not sandwiched, we can replace it by a principal partition by Lemma~\ref{innermost}.  
 We continue replacing innermost non-principal sides until we encounter one that is sandwiched (which must exist since $M(V)>M(L)$) and hence irreplaceable.   
 
As shown in  \cite{CSV}, the star of a Salvetti  $S_\Gamma$ in $K_\Gamma$ is the union of the cubes with $S_\Gamma$ as a vertex, and these cubes are identified with weakly compatible collections of \GW-partitions.  The stabilizer $S_\Gamma$ under the action of $\UG$ is isomorphic to the subgroup generated by graph automorphisms and inversions.
The effect of such an automorphism on the cubes in the star is to permute the labels of $V^\pm$.  Since incidence relations are preserved, any such automorphism  sends a \GW-partition to the ``same" \GW-partition with the labels permuted.   Since irreplaceable partitions are characterized by being sandwiched, such an automorphism sends sandwiched partitions to sandwiched partitions, and thus sends free faces to free faces.  Thus collapsing these free faces is an equivariant operation, giving an equivariant deformation retraction of $K_\Gamma$.  
\end{proof}

\end{document}